\documentclass{amsart}
\usepackage{amsmath}
\usepackage{graphicx} 
\usepackage{amsfonts} 
\usepackage{amssymb}
\usepackage{pdfsync}
\usepackage[font=small,format=hang,labelfont={sf,bf}]{caption}
\usepackage{color}
\usepackage{mathrsfs}
\usepackage{epsfig}
\usepackage{subfig}
\usepackage{url}
\usepackage{varioref}

\theoremstyle{plain}
\newtheorem{theorem}{Theorem}[section]

\newtheorem{lemma}[theorem]{Lemma}
\newtheorem{proposition}[theorem]{Proposition}

\newtheorem{remark}[theorem]{Remark}
\newtheorem{definition}[theorem]{Definition}
\theoremstyle{definition}
\theoremstyle{remark}
\numberwithin{equation}{section}

\newcommand{\ep}{\varepsilon}
\newcommand{\e}{\varepsilon}

\newcommand{\diam}{\mathrm{diam}\,}

\newcommand{\R}{\mathbb{R}}
\newcommand{\N}{\mathbb{N}}
\newcommand{\s}{\sigma}

\newcommand{\K}{\mathcal{K}}

\newcommand{\hh}{\hat H}
\newcommand{\jj}{\hat J}
\newcommand{\Hsc}{\mathscr{H}}
\newcommand{\Jsc}{\mathscr{J}}

\newcommand{\di}{\textrm{dist}}

\newcommand{\ud}{\;\mathrm{d}}

\newcommand{\Mgen}{\mathcal{M}(\R^d)}
\newcommand{\M}{\mathcal{M}_{\mathrm{f}}(\R^d)}

\newcommand{\weakly}{\rightharpoonup}           
\newcommand{\weakstar}{\stackrel{*}{\weakly}}   

\newcommand{\loc}{\mathrm{loc}}

\newcommand{\J}{\mathcal{J}}

\usepackage{latexsym}
\newcommand{\newatop}{\genfrac{}{}{0pt}{1}} 
\title[The $0$-fractional perimeter] {The $0$-fractional perimeter between  \\ fractional perimeters and Riesz potentials}
\author[L. De Luca]
{L. De Luca}
\address[Lucia De Luca]{Dipartimento di Matematica, Universit\`a di Pisa, Largo Bruno Pontecorvo 5, 56127 Pisa, Italy}
\email[L. De Luca]{lucia.deluca@unipi.it}
\author[M. Novaga]
{M. Novaga}
\address[Matteo Novaga]{Dipartimento di Matematica, Universit\`a di Pisa, Largo Bruno Pontecorvo 5, 56127 Pisa, Italy}
\email[M. Novaga]{matteo.novaga@unipi.it}
\author[M. Ponsiglione]
{M. Ponsiglione}
\address[Marcello Ponsiglione]{Dipartimento di Matematica ``Guido Castelnuovo'', Sapienza Universit\`a di Roma, P.le Aldo Moro 5, I-00185 Roma, Italy}
\email[M. Ponsiglione]{ponsigli@mat.uniroma1.it}
\begin{document}
\maketitle
\begin{abstract}
This paper provides a unified point of view on fractional perimeters and Riesz potentials. Denoting by $H^\s$ - for $\s\in (0,1)$ - the $\s$-fractional perimeter and by $J^\s$ - for $\s\in (-d,0)$ - the $\s$-Riesz energies acting on characteristic functions, we prove that both  functionals can be seen as limits of renormalized self-attractive energies as well as limits of repulsive interactions between a set and its complement. 

 We also show that the functionals $H^\s$ and $J^\s$\,, up to a suitable additive renormalization diverging when $\s\to 0$, belong to 
a continuous one-parameter  family of functionals, which for $\s=0$ gives back a new object we refer to as {\it $0$-fractional perimeter}. 
All the convergence results with respect to the parameter $\s$ and to the renormalization procedures are obtained in the framework of $\Gamma$-convergence. As a byproduct of our analysis, we obtain the isoperimetric inequality for the $0$-fractional perimeter. 
\end{abstract}
\tableofcontents

\section*{Introduction}
Given $\s\in (0,1)$\,, the $\s$-fractional perimeter \cite{CRS} of a measurable set $E\subseteq\R^d$  is defined as 
\begin{equation}\label{sperdef}
H^\s(E):=\int_{E}\int_{\R^d\setminus E} \frac{1}{|x-y|^{d+\s}}\ud y \ud x\, .
\end{equation}
Moreover, given $\s\in (-d,0)$ one can define   $\s$-Riesz energy functionals as
\begin{equation}\label{rieszendef}
J^\s(E) :=-\int_{E}\int_{E}\frac{1}{|x-y|^{d+\s}}\ud y\ud x\,.
\end{equation}

The main purpose of this paper is to introduce meaningful extensions of both functionals for all $\s\in(-d,1)$. Clearly, plugging  $\s\in (-d,0]$  in \eqref{sperdef}, as well as 
$\s\in [0,1)$ in \eqref{rieszendef},  would give back infinite tail and core energies, respectively. Moreover, for every set $E$ with $0<|E|<+\infty$  we have 
$$
\lim_{\s\to 0^+} H^\s(E)=+\infty, \qquad  \lim_{\s\to 0^-} J^\s(E)=-\infty \, .
$$  

A natural question is then to understand the blow up scaling of these energies as $\s\to 0$.
In \cite{DFPV, MS}, the asymptotics of the $\s$-fractional perimeter inside a regular domain $\Omega$, as $\s\to 0^+ $, has been studied. In \cite{DFPV}, the authors have proven that, under suitable conditions on the set $E$\,, the $\s$-fractional perimeter of $E$ 
scaled by $\s$, converge to $d\omega_d|E|$ as $\s\to 0^+$,
where $\omega_d$ denotes the measure of the unit ball in $\R^d$\,.

In this paper we provide a unified point of view on fractional perimeters and Riesz potentials and, in particular, we develop a ``first order'' $\Gamma$-convergence analysis (see \cite{BrTr}) of the functionals $H^\s$ and $ J^\s$ as $\s\to 0$. 
We first introduce suitable regularization procedures, usually referred to as the {\it core radius approach}, to cut off the tail and the core energy from $H^\s$ and $J^\s$ respectively, extending these functionals to all $\s\in(-d,1)$,
including $\s=0$, by setting
\begin{equation*}
\begin{aligned}
\displaystyle H_{\rho}^\s(E)&:=\int_{E}\int_{B_\rho(x)\setminus E}\frac{1}{|x-y|^{d+\s}}\ud y\ud x\,,\\
\displaystyle J_{\rho}^\s(E)&:=-\int_{E}\int_{E\setminus B_\rho(x)}\frac{1}{|x-y|^{d+\s}}\ud y\ud x\,.
\end{aligned}
\end{equation*} 

For $\s\in(0,1)$, the functionals $H^\s_\rho$ converge to the $\s$-fractional perimeters as $\rho\to +\infty$, 
while  for $\s\in (-d,0)$, the functionals $J^\s_\rho$
converge to $\s$-Riesz potentials as $\rho\to 0^+$. 
Clearly, in  the remaining range of parameters 
these functionals still diverge.

Then, we introduce suitable renormalized functionals,  removing a  tail energy from $H^\s_\rho$, and adding a core energy to $J^\s_\rho$.
More precisely, for every $\sigma\in (-d,1)$, for every $\rho>0$, and for every set $E\subset \R^d$ with finite measure we set
$$
\begin{aligned}
\hh^\s_\rho(E):=H^\s_\rho(E)-\gamma_\rho^\s|E|\,, \qquad \jj^\s_\rho(E):=J^\s_\rho(E)-\gamma_\rho^\s|E|\,,
\end{aligned}
$$
where the constant $\gamma_\rho^\s$ is defined in \eqref{sgammaR}.
These renormalized functionals $\hh^\s_\rho$ and $\jj^\s_\rho$ converge, for all $\s \in (-d,1)\setminus \{0\}$ to fractional and Riesz type functionals. 
Setting $\gamma^\s:=\frac{d\omega_d}{\s}$, for $\s\in (-d,1)\setminus\{0\}$, and 
$$
\hh^\s(E):=H^\s(E)-\gamma^\s|E|\,, \textrm{ for }\s\in (0,1)\,,\qquad \jj^\s(E):=J^\s(E)-\gamma^\s|E|\,,\textrm{ for }\s\in (-d,0)\,,
$$
we have that for $\s\in(-d,0)$,  $\hh^\s_\rho\to \jj^\s$ as $\rho\to +\infty$ (and clearly $\hh^\s_\rho\to \hh^\s$ for $\s\in(0,1)$)\,, while for $\s\in (0,1)$, $\jj^\s_\rho\to \hh^\s$ as $\rho\to 0^+$ (and clearly $\jj^\s_\rho\to \jj^\s$ for $\s\in(-d,0)$)\,.
 
 These two families $\hh^\s$ and $\jj^\s$ of renormalized energies  are separated by the limit case $\s=0$\,. Indeed, also for $\s=0$\,, the following limits exist 
 $$
 \hh^0(E):=\lim_{\rho\to +\infty}\hh^0_\rho(E)=\lim_{\rho\to 0^+}\jj^0_\rho(E) = H^0_1(E) + J^0_1(E) \,.
 $$ 
 We refer to the functional $\hh^0$ as {\it $0$-fractional perimeter} since this is formally the limit of $\hh^\s$ as $\s\to 0^+$. 
Indeed, we shall show that 
 $$
\lim_{\s\to 0^+}\hh^\s =  \lim_{\s\to 0^-}\jj^\s= \hh^0\,,
 $$ 
 where the limits are understood in the sense of $\Gamma$-convergence; therefore, we can set $\jj^0:=\hh^0$ and understand the $0$-fractional perimeter 
 also as a {\it $0$-Riesz functional}. 
 
This functional is closely related to the  notion of logarithmic laplacian  $(-\Delta)^L$ introduced in \cite{CW}. 
 There, the authors prove that $(-\Delta)^L$,  computed on regular enough functions, 
is  the pointwise limit, as $\s\to 0^+$, of  a suitable renormalization of the fractional laplacian $(-\Delta)^{\frac{\s}{2}}$.  
While \cite{CW} deals with the functional analytic framework of the operator $(-\Delta)^L$, our paper focuses on the geometric framework of characteristic functions, and specifically on the variational analysis of the $\s$-fractional perimeter as $\s\to 0$\,.   
  In fact, our analysis consists in a $\Gamma$-convergence approach to the
two-parameter families of functionals introduced above, showing that  $\hh^\s_R$ and $\jj^\s_r$ are continuous, in the sense of $\Gamma$-convergence, with respect to variations of all the parameters $\s\in (-d,1)$, $r \in [0,+\infty)$ and $R\in (0,+\infty]$ (see Theorems \ref{gammaconvHR} and \ref{gammaconvJsr}).

Our $\Gamma$-convergence results are completed with compactness properties for sequences with equi-bounded energy. 
 It is well known that families of equi-bounded sets of equi-bounded perimeter are pre-compact in $L^1$, and such property extends to fractional perimeters. Here we show the same compactness property also for the new $0$-fractional perimeter. In fact, we can deal also with the case of varying parameters $\s\in[0,1)$, and $R\in(0,+\infty]$
 (see Theorem \ref{compH}).
 Analogous compactness results hold for the functionals $\jj^\s_r$ for $\s\in[0,1)$ when $r\to 0^{+}$, for  $\hh^\s_R$ when  $\s\to 0^-$\,, and for $\jj^\s_r$ when $r\to 0^+$ and $\s\to 0^-$ simultaneously. In all the other cases we expect only weak$*$ compactness in the family of $L^1$ densities.
We address the interested reader to \cite{JW,AdP}, 
where the authors provide compactness results for nonlocal Sobolev spaces,
for a large class of non-integrable kernels.   
 
Summarizing, the main novelty  of our approach consists in casting fractional perimeters into the framework of self-attractive Riesz potentials, and viceversa. The underlying idea is that fractional perimeters,  defined through interaction potentials of the set with its complement, can be  formally seen  as  the opposite of the self-interaction of the set with itself, but with an infinite core energy. This heuristic point of view is formalized by our analysis through rigorous renormalization procedures. The advantage of this approach is that one can exploit classical techniques for self-attracting energies to the framework of fractional perimeters and to the new $0$-fractional perimeter. A clarifying example of this fact is given by the isoperimetric inequality. Indeed, for self-attractive interaction potentials the celebrated Riesz inequality states that the energy is maximized on radially symmetric functions and,  under $L^\infty$ constraint, by characteristic functions of balls. In the terminology of fractional perimeters, this is nothing but the fractional isoperimetric inequality, proven in \cite{FS,FMM, FFMMM}. Here we provide a self-contained proof based on Riesz inequality, which provides the $\s$-fractional isoperimetric inequality and its stability also in the limit case $\s=0$. We refer to \cite[Proposition 3.1]{CN} for a similar result in the case  of nonlocal perimeters with a
general radially symmetric interaction kernel.

Finally, we point out that the $0$-fractional perimeter fits into the class of nonlocal perimeters introduced in \cite{CMP}, up to the fact that it is, in general, non-positive. Therefore,
it would be interesting to study  the corresponding  $0$-fractional mean curvature flow. We notice that $\s$-fractional mean curvature flows (for $\s\in (0,1)$) are nowadays relatively well understood (see \cite{I,CMP}), and that their limit as $\s\to 1$ gives back the classical mean curvature flow \cite{I}. This is consistent with the fact that the $\s$-fractional perimeter, rescaled by $(1-\s)$, converges to $d \omega_d$ times the Euclidean perimeter (see \cite{ADPM}, for a $\Gamma$-convergence result, and the references therein). 
A natural problem is to study the limit of (suitably rescaled)  $\s$-fractional mean curvature flows as $\s\to 0^+$. On the one hand, one could expect that such flows, suitably reparametrized in time, converge to evolutions of sets with constant normal velocity; on the other hand, gradient flows of renormalized $\s$-fractional perimeters, as well as $\s$-fractional mean curvature flows with  a volume constraint, 
could converge to the $0$-fractional mean curvature flow, as $\s\to 0^+$. 


\section{Renormalized fractional perimeters and renormalized Riesz energies}
Let $\Mgen$ be the family of measurable sets in $\R^d$ and let
$$
\M:=\{E\in\Mgen\,:\,|E|<+\infty\}\, .
$$
For $\s\in (0,1)$\,,  the {\it $\s$-fractional perimeter} $H^{\s}\colon \Mgen\to [0,+\infty]$ is defined  in \eqref{sperdef}, while for every $\s\in (-d,0)$\,, the {\it $\s$-Riesz energy} $J^{\s}\colon \Mgen\to [-\infty,0]$ is defined in \eqref{rieszendef}.

For every $\s\neq 0$ we define  $\gamma^\s:=\frac{d\omega_d}{\s}$\,. 
Letting $E\in\M$,  we recall that  
\begin{equation}\label{spermoddef}
\hh^\s(E):  = H^\s(E)-\gamma^\s|E|\,,\qquad \jj^{\s}(E):=J^{\s}(E)-\gamma^{\s}|E|\,.
\end{equation}

\begin{remark}\label{codo}
\rm{
By definition, $\hh^\s\colon \M\to (-\infty,+\infty]$\,. Moreover, by Riesz inequality (see Theorem \ref{riesz}), we have
\begin{equation}\label{bound-s}
J^{\s}(E)\ge - \int_{B^{|E|}}\int_{B^{|E|}}\frac{1}{|x-y|^{d+\s}}\ud y\ud x>-\infty\,,
\end{equation}
where $B^{|E|}$ denotes the ball with center at $0$ and volume equal to $|E|$\,.
It follows that $\jj^{\s}\colon \M\to\R$\,. 
 }
 \end{remark}
 
We introduce two types of approximations of the functionals $\hh^\s$ and $\jj^{\s}$ above.
 
Let $\s\in (-d,1)$. 
For every $R>0$ we define the functionals $H^\s_R:\Mgen\to  [0,+\infty]$ as
\begin{equation*}
H^\s_R(E):=\int_{E} \int_{B_R(x)\setminus E} \frac{1}{|x-y|^{d+\s}}\ud y\ud x\,.
\end{equation*}
Moreover, for every $\rho>0$ we set
\begin{eqnarray}\label{sgammaR}
\gamma^{\s}_\rho:=\left\{\begin{array}{ll}
\displaystyle d\omega_d\frac{1-\rho^{-\sigma}}{\sigma}&\textrm{if }\sigma\neq 0\,,\\ 
\displaystyle d\omega_d\log \rho&\textrm{if }\sigma=0\,,
\end{array}
\right.
\end{eqnarray}
and for every $R >  0$ we introduce the functionals $\hat H^\s_R$ defined by
\begin{equation*}
\hh^\s_R(E): =  H_R^\s(E)-\gamma_R^\s|E| \,.
\end{equation*}
Notice that if $\s\in[0,1)$\,, then $\hh_R^\s:\M\to (-\infty,+\infty]$\,, whereas if $\s\in (-d,0)$ then $\hh_R^\s:\M\to \R$\,. 

Furthermore, for every  $r>0$ we define the functionals 
$J^\s_r: \Mgen \to [-\infty,0]$ as
\begin{equation}\label{Jr}
J^\s_r(E):= -\int_E  \int_{E \setminus B_r(x)} \frac{1}{|x-y|^{d+\s}} \ud y
\ud x\,
\end{equation}
and the renormalized functionals $\jj^\s_r:\M\to\R$ as
\begin{equation*}
\jj^\s_r(E):= J^\s_r(E)-\gamma_r^\s|E|, \quad \text{for all } E\in \M. 
\end{equation*}
\begin{remark}\label{lsc}
\rm
Let $R>0$\,.
By Fatou Lemma, it immediately follows that the functionals $H^\s,\,\hh^\s$ (for $\s\in (0,1)$) and  $H^\s_R,\, \hh^\s_R$ (for $\s\in (-d,1)$) are lower semicontinuous with respect to the strong $L^1$ convergence of characteristic functions. 
\end{remark}
\begin{lemma}\label{continuityJ}
Let $r>0$\,.
The functionals $J^{\s},\,\jj^{\s}$ (for $\s\in (-d,0)$) and $J_r^\s,\jj^\s_r$ (for $\s\in (-d,1)$) are continuous with respect to the strong $L^1$ convergence of characteristic functions.
\end{lemma}
\begin{proof}
We prove only the continuity of the functionals $J^\s_r$ and $\jj^\s_r$\,, being the proof for $J^{\s}$ and $\jj^{\s}$ fully analogous. Moreover we notice that the functionals $\jj^{\s}_r$ are nothing but a continuous perturbation of the functionals $J^\s_r$ so that it is enough to prove only the continuity of $J^\s_r$\,.
For all  $\eta_1,\, \eta_2 \in L^1(\R^d;[0,1])$\,, 
we set 
\begin{equation}\label{Jcal}
\J^\s_r (\eta_1,\eta_2):=   -\int_{\R^d}  \eta_1(x) \left[ \int_{\R^d \setminus B_r(x)} \frac{\eta_2(y)}{|x-y|^{d+\s}} \ud y
 \right]\ud x\,. 
\end{equation}
Clearly, $\J^\s_r$ is bilinear and continuous, i.e., 
\begin{equation}\label{contJsto}
|\J^\s_r (\eta_1,\eta_2)| \le r^{-(d+\s)} \|\eta_1\|_{L^1}  \, \|\eta_2\|_{L^1}\qquad\textrm{for all } 
\eta_1,\, \eta_2 \in L^1(\R^d;[0,1])\,.
\end{equation}
It follows that $\J^\s_r(\eta_n,\eta_n)\to \J^\s_r(\eta,\eta)$ as $n\to +\infty$ for every $\{\eta_n\}_{n\in\N}\subset L^1(\R^d;[0,1])$ converging to some $\eta\in L^1(\R^d;[0,1])$\,. 
Since 
\begin{equation*}
J^\s_r (E) = \J^\s_r(\chi_E,\chi_E)\qquad\textrm{for all }E\in\M\,,
\end{equation*}
we get the claim.
\end{proof}
\begin{remark}\label{cores}
\rm
We notice that, for every  $\s\in (-d,1)$\,, the renormalization constants introduced above can be seen either as core or tail energy terms; in fact, they are nothing but 
$$
\begin{aligned}
\gamma^\s &= \int_{\R^d\setminus B_1} \frac{1}{|z|^{d+\s}} \ud z&\quad\textrm{ if }\s\in(0,1),
\qquad\gamma^\s&= - \int_{B_1}\frac{1}{|z|^{d+\s}}\ud z &\qquad\textrm{ if }\s\in(-d,0),\\
\gamma^\sigma_R&= \int_{B_R\setminus B_1} \frac{1}{|z|^{d+\sigma}} \ud z &\quad\textrm{ if }R\ge 1,\qquad
\gamma^\sigma_r&= -\int_{B_1\setminus B_r} \frac{1}{|z|^{d+\sigma}} \ud z &\quad\textrm{ if }r\le1\,,
\end{aligned}
$$
where here and below $B_\rho:=B_\rho(0)$ for every $\rho>0$\,.
It follows that for every $0<\rho_1<\rho_2<+\infty$
\begin{equation}\label{funraggi}
\gamma^\s_{\rho_2}-\gamma_{\rho_1}^\s=\int_{B_{\rho_2}\setminus B_{\rho_1}}\frac{1}{|z|^{d+\sigma}} \ud z\,.
\end{equation}
\end{remark}
\begin{lemma}\label{fon}
For all $\s\in(-d,1)$ and for all $E\in\M$\,, the quantity 
\begin{equation}\label{somma00}
H^\s_\rho(E)+J^\s_\rho(E)-\gamma^\s_\rho |E|
\end{equation}
is independent of $\rho$\,.
Moreover, 
\begin{equation}\label{ovvieta}
\hh^\s=H^\s_1+J^\s_1\quad \textrm{ for }\s\in(0,1)\,,\qquad \jj^\s=H^\s_1+J^\s_1\quad \textrm{ for }\s\in(-d,0)\,.
\end{equation}
\end{lemma}
\begin{proof}
For every $0<r< R<+\infty$\,, we have
\begin{equation*}
\begin{aligned}
&H^\s_R(E)=H_r^\s(E)+\int_{E}\int_{(B_R(x)\setminus B_r(x))\setminus E}\frac{1}{|x-y|^{d+\s}}\ud y\ud x\\
=&H^\s_r(E)+\gamma_R^\s|E|-\gamma_r^\s|E|-\int_{E}\int_{E\setminus B_r(x)}\frac{1}{|x-y|^{d+\s}}\ud y\ud x+\int_{E}\int_{E\setminus B_R(x)}\frac{1}{|x-y|^{d+\s}}\ud y\ud x\\
=& H^{\s}_r(E)+\gamma_R^\s|E|-\gamma_r^\s|E|+J_r^\s(E)-J_R^\s(E)\,,
\end{aligned}
\end{equation*}
whence we deduce \eqref{somma00}.

Let us pass to the proof of  \eqref{ovvieta}. 
First, we notice that 
\begin{equation}\label{poslim}
\lim_{R\to +\infty} J^{\s}_R(E)=0 \quad \text{ for all } \s\in [0,1)\, , \qquad \lim_{r\to 0} H^{\s}_r(E)=0 \quad \text{ for all } \s\in(-d,0)\,.
\end{equation}
Then,  recalling also Remark \ref{cores},  \eqref{ovvieta}
 can be easily deduced by taking the limits as $\rho \to 0$ and $\rho\to +\infty$  in \eqref{somma00}.
\end{proof}
\section{Convergence of $\hh^\s_R$ as $R\to +\infty$}\label{one}
In this section, we establish the convergence of $\hh^\s_R$ as $R\to +\infty$\,.
We distinguish among three cases: $\s\in (-d,0)$\,, $\s\in (0,1)$\,, $\s=0$\,.

We start by discussing the case $\s\in (-d,0)$\,.
\begin{proposition}\label{-sconvdivR}
Let $\s\in (-d,0)$\,. For every $E\in\M$\,, $H^{\s}_R(E)$ and $\gamma_R^{\s}$ are monotonically increasing with respect to $R$ and
\begin{equation}\label{trivlim}
\lim_{R\to +\infty}H^{\s}_R(E)=\lim_{R\to +\infty}\gamma_R^{\s}=+\infty\,.
\end{equation}
Moreover, $\hh^{\s}_R(E)$ is monotonically non-increasing with respect to $R$
and 
\begin{equation}\label{limit-}
\jj^{\s}(E)=\lim_{R\to +\infty}\hh^{\s}_R(E)\, \qquad \text{ for all } E\in\M\,.
\end{equation}
Furthermore, the convergence in \eqref{limit-} is uniform with respect to  $\s$; precisely, for every $R>0$
\begin{equation}\label{rate0}
0\le \hh^{\s}_R(E)-\jj^{\s}(E)\le \frac{|E|^2}{R^{d+\s}} \qquad \text{ for all } E\in\M \,.
\end{equation}
Finally, $\jj^{\s}$ is the  $\Gamma$-limit of the functionals $\hh^{\s}_R$ as $R\to +\infty$\,, with respect to the strong $L^1$ topology.
\end{proposition}
\begin{proof}
The limits in \eqref{trivlim} are trivial consequences of the definitions of $H^{\s}_R(E)$ and $\gamma_R^{\s}$ for $\s\in (-d,0)$\,.
Let now $R>0$\,. By Remark \ref{cores}, we have
$$
\gamma_R^{\s}-\gamma^{\s}=\int_{B_R}\frac{1}{|z|^{d+\s}}\ud z,
$$
whence we deduce
\begin{equation}\label{quasi0}
\begin{aligned}
\hh^{\s}_R(E)=&-\int_{E}\int_{B_R(x)\cap E}\frac{1}{|x-y|^{d+\s}}\ud y\ud x-\gamma^{\s}|E|\\
=&\jj^{\s}(E)+\int_{E}\int_{E\setminus B_R(x)}\frac{1}{|x-y|^{d+\s}}\ud y\ud x\,;
\end{aligned}
\end{equation}
since the last integral is monotonically non-increasing with respect to $R$\,, 
the same holds for $\hh^{\s}_R(E)$; moreover,  by the monotone convergence Theorem, we deduce that $\hh^{\s}_R(E)$ converge to $\jj^{\s}(E)$\,.
By \eqref{quasi0}, it follows that
\begin{equation*}
\begin{aligned}
\hh^{\s}_R(E)-\jj^{\s}(E)=\int_{E}\int_{E\setminus B_R(x)}\frac{1}{|x-y|^{d+\s}}\ud y\ud x\le \frac{|E|^2}{R^{d+\s}}\,,
\end{aligned}
\end{equation*}
i.e., \eqref{rate0} holds.
Finally, the $\Gamma$-convergence of $\hh^{\s}_R$ to $\jj^{\s}$ is an obvious consequence of \eqref{rate0}.
\end{proof}
We now consider the case $\s\in (0,1)$. Both the result and its proof are  fully analogous to those of Proposition \ref{-sconvdivR};  we only provide the corresponding statement. 
\begin{proposition}\label{sconvdivR}
Let $\s\in (0,1)$.  For every $E\in\Mgen$,
$H^\s_R(E)$ is monotonically non-decreasing with respect to $R$ and tends to $H^\s(E)$ as $R\to +\infty$. 

Moreover, $\gamma_R^\s$  
is monotonically increasing with respect to $R$ and tends to $\gamma^\s$ as $R\to +\infty$.
As a consequence,  
\begin{equation}\label{limit}
\hh^\s(E)=\lim_{R\to +\infty}\hh^\s_R(E)\, \qquad \text{ for all } E\in\M.
\end{equation}
Furthermore, $\hh^\s_R(E)$ is monotonically non-increasing with respect to $R$, and the convergence in \eqref{limit} is uniform with respect to  $s$; precisely, for every $R>1$
\begin{equation*}
0\le \hh^\s_R(E)-\hh^\s(E)\le \frac{|E|^2}{R^{d+\s}} \qquad \text{ for all } E\in\M \,.
\end{equation*}
Finally, $\hh^\s$ is the  $\Gamma$-limit of the functionals $\hh^\s_R$ as $R\to +\infty$\,, with respect to the strong $L^1$ convergence of characteristic functions.
\end{proposition}
We finally introduce the $0$-fractional perimeter as the limit of the functionals $\hh^0_R$ as $R\to +\infty$\,.
\begin{proposition}\label{monotone}
For every $E \in \M$ the functionals $\hh^0_R(E)$ are monotonically non-increasing with respect to $R$ and 
\begin{equation} \label{fores}
  \lim_{R\to +\infty} \hh^0_R(E) =  H^0_1(E) + J^0_1(E) =: \hh^0(E) \, .
\end{equation}
Finally, $\hh^0$ is the  $\Gamma$-limit of the functionals $\hh^0_R$ as $R\to +\infty$\,, with respect to the strong $L^1$ convergence of characteristic functions.
\end{proposition}
\begin{proof}
Let $0<R_1<R_2<+\infty$\,; then, recalling \eqref{funraggi}, we have
\begin{eqnarray*}
\hh^0_{R_2}(E)&=&\hh^0_{R_1}(E)+\int_{E}\left[\int_{(B_{R_2}(x)\setminus B_{R_1}(x))\setminus E}\frac{1}{|x-y|^d}\ud y+\gamma_{R_1}^0-\gamma_{R_2}^0\right]\ud x\\
&\le& \hh^0_{R_1}(E)+\int_{E}\left[\int_{B_{R_2}(x)\setminus B_{R_1}(x)}\frac{1}{|x-y|^d}\ud y-d\omega_d\log\frac{R_2}{R_1}\right]\ud x\\
&=&\hh^0_{R_1}(E)\,.
\end{eqnarray*}
Therefore, $\hh^0_R(E)$ are monotonically non-increasing with respect to $R$. 
Moreover, by Lemma \ref{fon}, for every $R>0$, we have
$$
H^0_R(E)+J^0_R(E)-\gamma^0_R |E| = H^0_1(E)+J^0_1(E)\,,
$$
whence, sending $R\to +\infty$ and recalling \eqref{poslim}
we deduce \eqref{fores}.
Finally, by Remark \ref{lsc} and by Lemma \ref{continuityJ}, we have that the functional $\hh^0$ is lower semicontinuous with respect to the strong $L^1$ convergence of the characteristic functions; therefore, by the monotonicity of $\hh^0_R$ with respect to $R$ and by \cite[Proposition 5.7]{DalM},  we deduce the $\Gamma$-convergence result.
\end{proof}
\begin{lemma}
For every $E\in \M$ we have 
$$
\hh^0(E) \ge - |E| \omega_d \log \Big( \frac{|E|}{\omega_d} \Big ).
$$
\end{lemma}
\begin{proof}
For every $R>0$ we set $r_{R,E}(x):=(\frac{|E\cap B_R(x)|}{\omega_d})^{\frac{1}{d}}$\,. By Lemma \ref{rieszann0}, we have
$$
\begin{aligned}
H^0_R(E)&=\int_{E} \int_{B_R(x)\setminus E} \frac{1}{|x-y|^{d}}\ud y\ud x 
\\
&\ge \int_{E} \int_{B_R(x)\setminus B_{r_{R,E}(x)}(x)} \frac{1}{|x-y|^{d}}\ud y\ud x = \int_{E} d \omega_d \log \frac{R}{r_{R,E}(x)}\ud x\\
&=\gamma_R^0|E|-\omega_d\int_{E}\log\frac{|E\cap B_R(x)|}{\omega_d}\ud x\,,
\end{aligned}
$$
whence the claim follows by sending $R\to +\infty$ and using Proposition \ref{monotone} and the monotone convergence Theorem.
\end{proof}
\begin{definition}\label{defdue}
\rm
We refer to the functional $\hh^0: \M \to (-\infty, +\infty]$  
as {\it $0$-fractional perimeter}. 
\end{definition}
\noindent We observe that, as a consequence the definition above and \cite[Eq. (1.5)]{CW}, 
we can write the $0$-fractional perimeter of $E$ as
\[
\hh^0(E) = \int_E  (-\Delta)^L \chi_E(x)\ud x \,,
\]
where  $(-\Delta)^L$ denotes the logarithmic laplacian introduced in \cite{CW}.


\begin{remark}\label{monosi}
\rm
By \eqref{ovvieta} and \eqref{fores} it immediately follows that, for every $E\in\M$ with positive measure, $\hh^\s(E)$ is increasing with respect to $\s\in[0,1)$. 
\end{remark}

\begin{remark}\label{rmkbded}
\rm
If $E$ is a bounded set, by arguing as in the  proof of Proposition \ref{monotone}, one immediately has
$$
\hh^0_{R_1}(E)=\hh^0_{R_2}(E)\qquad\textrm{for every }R_2>R_1>\diam (E)\,,
$$
whence
$$
\hh^0(E)=\hh^0_{R}(E)\qquad\textrm{for every }R>\diam (E)\,.
$$
Analogously, one can show that if $E$ is a bounded set, 
for every  
$R>\diam (E)$ it holds
$$
\jj^{\s}(E)=\hh^{\s}_R(E)\quad (\textrm{for }\s\in (-d,0))\qquad\textrm{and}\qquad
\hh^{\s}(E)=\hh_R^{\s}(E)\quad (\textrm{for }\s\in (0,1))\,.
$$
\end{remark}
\begin{lemma}\label{Rsubmodlemma}
Let $\s\in (-d,1)$\,. For every $R>0$ the functionals $H^\s_R$  are submodular, i.e., for every $E_1,E_2\in\M$
\begin{equation}\label{RsubmodineqH}
H_R^\s(E_1\cup E_2)+H_R^\s(E_1\cap E_2)\le H^\s_R(E_1)+H^\s_R(E_2)\,.
\end{equation}
As a consequence, also the functionals $\hh_R^\s$ are submodular.

Moreover, the functionals $\hh^\s$ are submodular for $\s\in [0,1)$ and the functionals $\jj^\s$ are submodular for $\s\in (-d,0)$\,.
\end{lemma}

\begin{proof}
Fix $R>0$ and let $E_1,E_2\in\M$\,.
Trivially,  
\begin{equation*}
|E_1\cup E_2|+|E_1\cap E_2|=|E_1|+|E_2|.
\end{equation*}
Therefore, once  \eqref{RsubmodineqH} is proven, the submodularity of $\hh_R^\s$ follows, and in turn the submodularity of $\hh^\s$ and $\jj^\s$, by sending $R\to +\infty$ and using Propositions \ref{-sconvdivR}, \ref{sconvdivR} and  \ref{monotone}.

To prove \eqref{RsubmodineqH} we preliminarily notice that for every disjoint sets $A_1,A_2\in\M$ and for every $x\in\R^d$ we have
\begin{equation*}
\begin{aligned}
B_R(x)\cap A_2 \subseteq B_R(x)\setminus A_1, \quad B_R(x)\setminus (A_1\dot{\cup} A_2)=(B_R(x)\setminus A_1)\setminus (B_R(x)\cap A_2)
\end{aligned}
\end{equation*}  
(where $\dot{\cup}$ denotes the disjoint union), 
so that
\begin{equation}\label{fundacomp}
\begin{aligned}
H^\s_R(A_1\dot\cup{A_2})=&H_R^\s(A_1)-\int_{A_1}\int_{B_R(x)\cap A_2}\frac{1}{|x-y|^{d+\s}}\ud y\ud x\\
&+H^\s_R(A_2)-\int_{A_2}\int_{B_R(x)\cap A_1}\frac{1}{|x-y|^{d+\s}}\ud y\ud x\,.
\end{aligned}
\end{equation}
Since
$$
E_1\cup E_2=(E_2\setminus E_1)\dot{\cup}E_1\,,
$$
by \eqref{fundacomp}, we deduce
\begin{equation}\label{E1UE2esp}
\begin{aligned}
H_R^\s(E_1\cup E_2)=& H_R^\s(E_2\setminus E_1)-\int_{E_2\setminus E_1}\int_{B_R(x)\cap E_1}\frac{1}{|x-y|^{d+\s}}\ud y\ud x\\
&+H_R^\s(E_1)-\int_{E_1}\int_{B_R(x)\cap (E_2\setminus E_1)}\frac{1}{|x-y|^{d+\s}}\ud y\ud x\,.
\end{aligned}
\end{equation}
Analogously, since 
$$
E_2=(E_2\setminus E_1)\dot{\cup}(E_1\cap E_2)\,,
$$
then, again by \eqref{fundacomp}, we get
\begin{equation}\label{E2esp}
\begin{aligned}
H_R^\s(E_2)=&H_R^\s(E_2\setminus E_1)
-\int_{E_2\setminus E_1}\int_{B_R(x)\cap(E_1\cap E_2)}\frac{1}{|x-y|^{d+\s}}\ud y\ud x\\
&+H^\s_R(E_1\cap E_2)-\int_{E_1\cap E_2}\int_{B_R(x)\cap(E_2\setminus E_1)}\frac{1}{|x-y|^{d+\s}}\ud y\ud x\,.
\end{aligned}
\end{equation}
In conclusion, by \eqref{E1UE2esp} and \eqref{E2esp}, we obtain 
\begin{align*}
&H_R^\s(E_1\cup E_2)+H_R^\s(E_1\cap E_2)=H_R^\s(E_1)+H^\s_R(E_2)\\
&+\int_{E_2\setminus E_1}\int_{B_R(x)\cap(E_1\cap E_2)}\frac{1}{|x-y|^{d+\s}}\ud y\ud x+\int_{E_1\cap E_2}\int_{B_R(x)\cap(E_2\setminus E_1)}\frac{1}{|x-y|^{d+\s}}\ud y\ud x\\
&-\int_{E_2\setminus E_1}\int_{B_R(x)\cap E_1}\frac{1}{|x-y|^{d+\s}}\ud y\ud x-\int_{E_1}\int_{B_R(x)\cap (E_2\setminus E_1)}\frac{1}{|x-y|^{d+\s}}\ud y\ud x\\
\le& H^\s_R(E_1)+H^\s_R(E_2)\,,
\end{align*}
i.e., \eqref{RsubmodineqH} holds.
\end{proof}

For every $\s\in (-d,1)$ we extend  the functional $H_R^\s$ to $L^1$ functions, obtaining a {\it generalized  total variation functional} $TV_{H^\s_R} :L^1(\R^d)\to [0,+\infty]$ (see \cite{Vis1, Vis2}) defined by 
\begin{equation}\label{tv}
TV_{H^\s_R} (u) := \int_{-\infty}^{+\infty} H^\s_R (\{ u>t\}) \ud t \qquad \text{ for all } u\in L^1(\R^d).
\end{equation}
By Lemma \ref{Rsubmodlemma} and by \cite[Proposition 3.4]{CGL}, $TV_{H^\s_R}$ is convex.
Analogously, one can consider for all $\s\in (-d,1)$ the 
functionals $TV_{\hh^\s} :L_{\mathrm{c}}^1(\R^d)\to (-\infty,+\infty]$ defined by 
\begin{equation*}
TV_{\hh^\s} (u) := \int_{-\infty}^{+\infty} \hh^\s (\{ u>t\}) \ud t \qquad \text{ for all } u\in L_{\mathrm{c}}^1(\R^d)\,,
\end{equation*}
where $L^1_{\mathrm{c}}(\R^d)$ denotes the set of $L^1$ functions compactly supported in $\R^d$\,.
\section{Convergence of $\jj^\s_r$ as $r\to 0^+$} \label{two}

In this section, we study the convergence of the functionals $\jj^\s_r$ as $r\to 0^+$\,.
We preliminarily notice that for every $\s\in (-d,1)$ and for every $r>0$
\begin{equation}\label{jjrexp}
\jj^\s_r(E)=\int_{E}j^\s_r(x,E)\ud x, \qquad \text{where } j_r^\s(x,E):=-\int_{E\setminus B_r(x)}\frac{1}{|x-y|^{d+\s}}\ud y-\gamma_r^\s \, .
\end{equation}
\begin{lemma}\label{rmonotone}
Let $\s\in(-d,1)$. For every $E \in \M$, the functions $j_r^\s(\cdot,E):\R^d\to\R$, as well as the functionals $\jj^\s_r(E)$, are monotonically non-increasing with respect to $r$\,. In particular, for every $x\in\R^d$ there exists $j^\s(x,E):= \lim_{r\to 0^+}j_r^\s(x,E)$ and there exists the limit
\begin{equation}\label{slimitex0}
\lim_{r\to 0^+} \jj_r^\s(E)=\int_{E}j^\s(x,E)\ud x\,.
\end{equation}
Moreover, if $\s\in (-d,0)$\,, then
\begin{equation}\label{-slimitex0}
j^{\s}(x,E)=-\int_{E}\frac{1}{|x-y|^{d+\s}}\ud y-\gamma^{\s}\,,
\end{equation}
and 
$$
\jj^{\s}(E)=\lim_{r\to 0^+} \jj_r^{\s}(E)\,,
$$
where $\jj^{\s}$ is the functional defined in \eqref{spermoddef}.
\end{lemma}
\begin{proof}
Let $0<r_1<r_2 <+\infty$\,.
 By the very definition of $j_r^\s(\cdot,E)$ and by \eqref{funraggi}, for every $x\in\R^d$ we have
\begin{eqnarray*}
j^\s_{r_1}(E)&=&j^\s_{r_2}(E)-\int_{(B_{r_2}(x)\setminus B_{r_1}(x))\cap E}\frac{1}{|x-y|^{d+\s}}\ud y+\gamma_{r_2}^\s-\gamma_{r_1}^\s\\
&\ge& j^\s_{r_2}(E)-\int_{B_{r_2}(x)\setminus B_{r_1}(x)}\frac{1}{|x-y|^{d+\s}}\ud y+\gamma_{r_2}^\s-\gamma_{r_1}^\s\\
&=&j^\s_{r_2}(E)\,.
\end{eqnarray*}
Therefore, $j_r^\s(x,E)$ monotonically converge to some $j^\s(x,E)$ for every $x\in \R^d$\,. 
Moreover, by \eqref{jjrexp} and by the monotone convergence Theorem, we deduce \eqref{slimitex0}.
Finally, \eqref{-slimitex0} is again a consequence of the monotone convergence Theorem and of the very definition of $\jj^{\s}$\,.
\end{proof}

\begin{definition}\label{rmonotonecons}
\rm
Thanks to Lemma \ref{rmonotone} we can extend the definition of $\jj^{\s}(E)$ also to the case $\s \in [0,1)$,  by setting  $\jj^{\s}(E) := \lim_{r\to 0^+} \jj_r^\s(E)$ for all $E\in\M$, as in \eqref{slimitex0}.
In Remark \ref{boundednessJ} below we show that actually $\jj^\s:\M\to (-\infty,+\infty]$.
\end{definition}
\begin{remark}\label{boundednessJ}
\rm{
Let $\s\in (-d,1)$\,, $E\in\M$\,, and $r\in(0,1]$\,; by the very definition of $J^\s_r$ in \eqref{Jr}, we have
\begin{equation}\label{boundJ}
|J^\s_r(E)|=-J^\s_r(E)\le \frac{|E|^2}{r^{d+\s}}\,,
\end{equation}
and hence
\begin{align}\label{boundjj}
\jj^\s_r(E)\ge -\frac{|E|^2}{r^{d+\s}}-\gamma_r^\s|E|\,.
\end{align}
Therefore, by \eqref{boundjj} and by Lemma \ref{rmonotone}, we deduce
\begin{align}\label{boundJJ}
\jj^\s(E)\ge \jj^\s_1(E)=J^\s_1(E)\ge -|E|^2\,.
\end{align}
Notice that the lower bound in \eqref{boundJJ} is worse than the one obtained in \eqref{bound-s} for $\s\in (-d,0)$\,. Nevertheless, such a lower bound is enough to guarantee that $\jj^\s:\M\to (-\infty,+\infty]$ for every $\s\in [0,1)$\,.
}
\end{remark}

Now  we extend the functionals $\jj^\s_r$ and $\jj^\s$  to $L^1$ densities 
 by setting, for all  $\rho \in L^1(\R^d;[0,1])$,
 \begin{equation}\label{jjrexpro}
 \begin{aligned}
&\jj^\s_r(\rho)=\int_{\R^d}\rho(x)j^\s_r(x,\rho)\ud x,\\
&\text{ where } j_r^\s(x,\rho):=-\int_{\R^d\setminus B_r(x)}\frac{\rho(y)}{|x-y|^{d+\s}}\ud y-\gamma_r^\s \, .
\end{aligned}
\end{equation}
Arguing as in the proof of Lemma \ref{rmonotone}, one can prove the following result.
\begin{lemma}\label{rmonotonel1}
Let $\s\in(-d,1)$; for every $\rho \in L^1(\R^d;[0,1])$, the functions $j_r^\s(\cdot,\rho):\R^d\to\R$, as well as the functionals $\jj^\s_r(\rho)$, are monotonically non-increasing with respect to $r$\,. In particular, for every $x\in\R^d$ there exists $j^\s(x,\rho):= \lim_{r\to 0^+}j_r^\s(x,\rho)$ and there exists
\begin{equation*}
\jj^\s(\rho):= \lim_{r\to 0^+} \jj_r^\s(\rho)=\int_{\R^d}\rho(x)j^\s(x,\rho)\ud x\,.
\end{equation*}
Moreover, if $\s\in (-d,0)$\,, then
\begin{equation*}
j^{\s}(x,\rho)=-\int_{\R^d}\frac{\rho(y)}{|x-y|^{d+\s}}\ud y-\gamma^{\s}\,.
\end{equation*}
\end{lemma}
By arguing as in Remark \ref{boundednessJ} we have that $\jj_r^\s(\rho)$ and $\jj^\s(\rho)$ are bounded from below by  $-\|\rho\|_{L^1}^2$. 
Moreover,  in Remark \ref{infinitness} we will see that if $\s\in [0,1)$ then $\jj^\s(\rho)=+\infty$ whenever $\rho$ is not the characteristic function of a set with finite measure. 
\begin{proposition}\label{cont}
Let $\s\in(-d,1)$;  for every $r>0$\,, the functionals $\jj^\s_r:L^1(\R^d;[0,1]) \to (-\infty,+\infty]$ are continuous with respect to the strong $L^1$ convergence.
As a consequence, their monotone limit $\jj^\s$ is 
lower semicontinuous; more precisely, $\jj^\s$ is the $\Gamma$-limit of $\{\jj^\s_r\}_{r>0}$ with respect to the strong $L^1$ topology, as $r\to 0^+$.
The same $\Gamma$-convergence result holds true for the functionals $\jj^\s_r$, $\jj^\s$ defined on $\M$.  
\end{proposition}
\begin{proof}
By arguing verbatim as in the proof of Lemma \ref{continuityJ} (see formulas \eqref{Jcal} and \eqref{contJsto}), one can prove the continuity of the functionals $\jj^\s_r$ with respect to the strong $L^1$ convergence.
Moreover, it is well known that monotone convergence of continuous functionals implies $\Gamma$-convergence to the pointwise (lower semicontinuous) limit \cite{DalM}. 
\end{proof}
\section{$\hh^\s = \jj^\s$}\label{H=J}

In view of Proposition \ref{-sconvdivR} and Lemma \ref{rmonotone} we have that for every $\s\in (-d,0)$ and for every $E\in\M$
$$
\lim_{R\to +\infty}\hh^{\s}_R(E)=\jj^{\s}(E)=\lim_{r\to 0^+}\jj^{\s}_r(E)\,,
$$
where $\jj^{\s}(E)$ is finite for every $E\in\M$\,.
By definition we set $\hh^{\s} := \jj^{\s}$ for every $\s\in (-d,0)$. 
Moreover, we recall that $\hh^{\s} := \jj^{\s} =  H^\s_1 + J^\s_1$  for every $\s\in (-d,0)$\,.
In this section we show that the identities above extends also to the  functionals $\hh^{\s}$ and $\jj^{\s}$ for $\s\in[0,1)$\,. More precisely, we prove that for every $\s\in[0,1)$ the functionals $\hh^\s$ and $\jj^\s$ coincide on all the measurable sets with finite measure and are finite on smooth sets.
\begin{theorem}\label{equfin}
Let $\s\in [0,1)$\,. For every $E\in\M$, it holds
\begin{equation*}
\jj^\s(E)=\hh^\s(E)\,. 
\end{equation*}
Moreover, $\jj^\s(E)$ and $\hh^\s(E)$ are finite if and only if  $H^\s_1(E)<+\infty$\,.
\end{theorem}
\begin{proof}
We distinguish among two cases.

{\it Case 1: $\hh^\s(E)=+\infty$\,.} By Proposition \ref{sconvdivR} and by Proposition \ref{monotone}, we have that $\hh^\s_1(E)=+\infty$ which, by the monotone convergence Theorem, implies
\begin{equation}\label{mono}
\lim_{r\to 0^+}\int_{E}\int_{(B_1(x)\setminus B_r(x))\setminus E}\frac{1}{|x-y|^{d+\s}}\ud y\ud x=H^\s_1(E)=\hh^\s_1(E)=+\infty\,.
\end{equation}
Moreover, by Remark \ref{cores}  for every $r\in(0,1]$ we have
\begin{equation}\label{startinf}
\begin{aligned}
&\int_{E}\int_{(B_1(x)\setminus B_r(x))\setminus E}\frac{1}{|x-y|^{d+\s}}\ud y\ud x\\
=&-\gamma_r^\s|E|-\int_{E}\int_{E\cap(B_1(x)\setminus B_r(x))}\frac{1}{|x-y|^{d+\s}}\ud y\ud x
\\
=&-\gamma_r^\s|E|-\int_{E}\int_{E\setminus B_r(x)}\frac{1}{|x-y|^{d+\s}}\ud y\ud x+\int_{E}\int_{E\setminus B_1(x)}\frac{1}{|x-y|^{d+\s}}\ud y\ud x\\
=&\jj^\s_r(E)- \jj_1^\s(E) \, .
\end{aligned}
\end{equation}
Therefore, by taking the limit as $r\to 0^+$ in \eqref{startinf}, using \eqref{mono}, Lemma \ref{rmonotone} and \eqref{boundJ}, we deduce that $\jj^\s(E)=+\infty$\,.

{\it Case 2: $\hh^{\s}(E)<+\infty$\,.}
By Proposition \ref{sconvdivR} and by Proposition \ref{monotone}, we have that there exists $R_1\ge 1$ such  that 
\begin{equation}\label{hplim}
\hh^\s_R(E)\le \hh^\s(E)+1\quad\textrm{ for } R\ge R_1\,.
\end{equation}
Let now $r\le 1$ and $R\ge R_1$\,; then, by \eqref{hplim} and Lemma \ref{fon}, we get
\begin{equation}\label{somma}
\begin{aligned}
\hh^{\s}(E)+1\ge& \hh^{\s}_R(E)=\jj^{\s}_r(E)+H^{\s}_r(E)-J^{\s}_R(E)\ge \jj^{\s}_1(E)+H^{\s}_r(E)\\
\ge&-|E|^2+H^{\s}_r(E)\,,
\end{aligned}
\end{equation}
where we have used also that $-J^{\s}_R(E)\ge 0$\,, Lemma \ref{rmonotone} and \eqref{boundJJ}\,;
it follows that
\begin{equation}\label{forconc}
H^{\s}_r(E)\le |E|^2+ \hh^{\s}(E)+1<+\infty\quad\textrm{ for every }0<r\le 1\,.
\end{equation}
By  the non-negativity and the monotonicity  of $H^{\s}_r(E)$ with respect to $r$\,, we deduce that there exists 
$$
\lim_{r\to 0^+}H^{\s}_r(E)\in [0,+\infty)\,.
$$
Let now $0<\bar r\le 1$\,; by the monotone convergence Theorem, we have
\begin{align*}
H^{\s}_{\bar r}(E)-\lim_{r\to 0^+}H_r^{\s}(E)&=\lim_{r\to 0^{+}}(H^{\s}_{\bar r}(E)-H^{\s}_{r}(E))\\
&=\lim_{r\to 0^{+}}\int_E\int_{(B_{\bar r}(x)\setminus B_{r}(x))\setminus E}\frac{1}{|x-y|^{d+\s}}\ud y\ud x\\
&=H^{\s}_{\bar r}(E)\,,
\end{align*}  
i.e., 
\begin{equation}\label{0esse}
\lim_{r\to 0^+}H_r^{\s}(E)=0\,.
\end{equation}
In conclusion, by taking first the limit as $r\to 0^+$ and then the limit as $R\to +\infty$ in the equality in \eqref{somma},  by  Lemma \ref{rmonotone}, \eqref{0esse},  Proposition \ref{monotone} and \eqref{poslim}, we deduce that $\hh^{\s}(E)=\jj^{\s}(E)$\,. 

Finally, we prove the last sentence in the statement. 
If $\hh^{\s}(E)=+\infty$, then  $H^{\s}_1(E)=\hh^{\s}_1(E) \ge \hh^{\s}(E)=+\infty$\,, whereas, if $\hh^{\s}(E)<+\infty$, then \eqref{forconc} with $r=1$ yields $H^{\s}_1(E)<+\infty$.
\end{proof}
It is well-known that fractional perimeters are finite on smooth sets. Next, we extend this property to the $0$-fractional perimeter.
\begin{proposition}\label{CF}
Let $\s\in [0,1)$\,.
If $E$ is an open bounded set with boundary of class $C^2$, then $\hh^{\s}(E)=\jj^{\s}(E)<+\infty$\,.\end{proposition}
\begin{proof}
In view of Theorem \ref{equfin} it is enough to show that $H^{\s}_1(E)<+\infty$\,.
Let $ T:= \frac 12 \|\Hsc\|_{L^\infty(\partial E)}^{-1}$, where $\Hsc$ is the second fundamental form.
Moreover, for all $t>0$ we set
\begin{equation}\label{superl}
E_t:= \{x\in E : \, \di(x,\partial E) > t\}. 
\end{equation}
Since $E$ has boundary of class $C^2$, we have that $T<+\infty$\,; moreover, there exists $C>0$ such that  $\mathcal H^{d-1}(\partial E_t) \le C$ and $E_t$ has boundary of class $C^2$ for all $t\in(0,T)$.
Then, for $0<r<\min\{1,T\}$ by coarea formula we have 
$$
\begin{aligned}
H^{\s}_r(E)&=\int_E\int_{B_r(x)\setminus E} \frac{1}{|x-y|^{d+\s}} \ud y\ud x 
=
\int_{E\setminus E_r}\int_{B_r(x)\setminus E} \frac{1}{|x-y|^{d+\s}} \ud y\ud x
\\
&= \int_0^r \int_{\partial E_t} \int_{B_r(x)\setminus E} \frac{1}{|x-y|^{d+\s}} \ud y\ud x \ud t\\
&\le
 \int_0^r \int_{\partial E_t} \int_{B_r(x)\setminus B_t(x)} \frac{1}{|x-y|^{d+\s}} \ud y\ud x \ud t
 \\
 &\le C  \int_0^r (G^{\s}(r) - G^{\s}(t)) \ud t <+\infty,
\end{aligned}
$$
where $G^{\s}(\tau)$ is the primitive of $\tau^{-\s}$ and in the last inequality we have used that it is  
integrable around the origin. By Propositions \ref{sconvdivR} and \ref{monotone}, we conclude that
$$
H^{\s}_1(E) = \hh^{\s}_1(E) \le \hh_r^{\s}(E) = H_r^{\s}(E) - \gamma_r^{\s} |E| <+\infty.
$$
\end{proof}
\section{Compactness}
This section is devoted to the proof of compactness results for the functionals $\jj^\s,\, \jj^\s_r,\,\hh^\s_R$\,.
First we  prove compactness 
properties for the functionals $\jj^{\s}_r$ and $\jj^{\s}$ for $\s\in [0,1)$.

\begin{theorem}[Compactness]\label{col} 
Let $\{\s_n\}_{n\in\N}\subset [0,1)$  and let $r_n\to 0^+$.
Let $U\subset\R^d$ be an open bounded set and let $\{E_n\}_{n\in\N}\subset \M$ be such that $E_n\subset U$  for all $n\in\N$\,. Finally, Let $C>0$\,.

If $\jj_{r_n}^{\s_n}(E_n)\le C$  for all $n\in\N$,
then, up to a subsequence, $\chi_{E_n} \to \chi_E$ in $L^1(\R^d)$ for some $E\in \M$.   

In particular, if $\jj^{\s_n}(E_n)\le C$ for all $n\in\N$\,,
then, up to a subsequence, $\chi_{E_n} \to \chi_E$ in $L^1(\R^d)$ for some $E\in \M$.   
\end{theorem}
\begin{proof}
Recalling the definition of $j^{\s}_r$ in \eqref{jjrexpro}, we claim the following two properties satisfied by all $\eta\in L^1(\R^d;[0,1])$:
\begin{itemize}
\item[(1)] For every  $x\in \R^d$, $r\in(0,1)$ it holds
$
j^{0}_r(x,\eta)\ge-\|\eta\|_{L^1}\,. 
$
\\
\item[(2)] 
For every Lebesgue point $x\in \R^d$ with Lebesgue value $\lambda \in (0,1)$ it holds
$$
\lim_{r \to 0^+}  j^{0}_r(x,\eta) = +\infty.
$$
\end{itemize}

\noindent
{\it Proof of (1).} For every $ r \in (0,1]$ we write
\begin{equation}\label{to-inf}
\begin{aligned}
j^{0}_r(x,\eta)=  - \int_{\R^d \setminus B_1(x)} \frac{\eta(y)}{|x-y|^{d}} \ud y  
 - \left[ \int_{B_1(x) \setminus B_r(x)} \frac{\eta(y)}{|x-y|^{d}} \ud y
+ \gamma^{0}_r \right].    
\end{aligned}
\end{equation}
By Remark \ref{cores} the last term in square brackets is always non-positive, whence property (1) easily follows. 

\noindent
{\it Proof of (2).}  We have to show that,   
whenever the Lebesgue value $\lambda$ of $\eta$
at $x$ is in $(0,1)$, the last term in square brackets in \eqref{to-inf} in fact tends to $ - \infty$ as $r \to 0^+$. To this purpose, in order to short notation we assume $x=0$, we let $\theta\in (0,1)$ be defined by 
$ \theta^d = \frac {1-\lambda}{2}$, and for all $k\ge 1$ we set $A^k := B_{\theta^{k-1}} \setminus B_{\theta^{k}}$. Since $\lambda$ is the Lebesgue value of $\eta$ at $0$, there exists $\bar k \in\N$ such that, for all $k > \bar k$ we have
$$
\frac{1}{\omega_d \theta^{(k-1)d}} \int_{B_{\theta^{k-1}}} \eta(y) \ud y \le \lambda+\frac{1-\lambda}{4}= \frac{1+3\lambda}{4}.
$$  
It follows that, for all $k > \bar k$, 
$$
\int_{A^k} \eta(y) \ud y \le \int_{B_{\theta^{k-1}}} \eta(y) \ud y \le \omega_d \frac{1+3\lambda}{4} \theta^{(k-1)d} =: m_k. 
$$
Now, we apply Lemma \ref{rieszannulus} with $m$ replaced by $m_k$, $s$ replaced by $\theta^{k}$ and in turn $R(m,s)$ replaced by $R(m_k,\theta^{k})$.
Therefore, setting $\hat A^k:=B_{R(m_k,\theta^{k})}\setminus B_{\theta^k}$\,, for all $k > \bar k$ we have
\begin{equation}\label{usolemma}
 \int_{A^k} \frac{\eta(y)}{|y|^{d}} \ud y
\le  \int_{\hat A^k} \frac{1}{|y|^{d}} \ud y\,.
\end{equation}
Now we prove that there exists  $\delta_{d,\lambda}>0$ (independent of $k$) such that
\begin{equation}\label{unifbound}
 \int_{\hat A^k} \frac{1}{|y|^{d}} \ud y
+ \gamma^{0}_{\theta^{k}}-\gamma_{\theta^{k-1}}^{0}\le -\delta_{d,\lambda}\,.
\end{equation}
By the very definition of $R(m_k,\theta^k)$ in Lemma \ref{rieszannulus}, we have that $|\hat A_k|=m_k$ so that
\begin{equation}\label{defRm}
R(m_k,\theta^k)=\theta^{k-1}\left(\theta^d+\frac{1+3\lambda}{4}\right)^\frac{1}{d}=\theta^{k-1}\left(\frac{\lambda+3}{4}\right)^{\frac 1 d}\,.
\end{equation}
By using \eqref{defRm}, we deduce \eqref{unifbound}  as follows:
\begin{equation*}
\begin{aligned}
 \int_{\hat A^k} \frac{1}{|y|^{d}} \ud y
+ \gamma^0_{\theta^{k}}-\gamma_{\theta^{k-1}}^0 & =d\omega_d\log\frac{R(m_k,\theta^k)}{\theta^k}+d\omega_d\log\theta
\\
& =\omega_d\log\left(\frac{\lambda+3}{4}\right)=:-\delta_{d,\lambda}\,,
\end{aligned}
\end{equation*}

Therefore, by \eqref{usolemma}, \eqref{unifbound} and by the fact that $\gamma^{0}_1=0$, for all $K > \bar k$ we get 
\begin{equation*}
\begin{aligned}
&\int_{B_1 \setminus B_{\theta^K}}\frac{\eta(y)}{|y|^{d}}\ud y+\gamma^{0}_{\theta^K}\\
&= \sum_{k=1}^{\bar k}\Big( \int_{A^k} \frac{\eta(y)}{|y|^{d}} \ud y+\gamma^{0}_{\theta^{k}}-\gamma_{\theta^{k-1}}^{0} \Big)
+\sum_{k=\bar k +1}^K\Big( \int_{A^k} \frac{\eta(y)}{|y|^{d}} \ud y+ \gamma^{0}_{\theta^{k}}-\gamma_{\theta^{k-1}}^{0} \Big)
\\
& \le \sum_{k=\bar k +1}^K
\Big( \int_{A^k} \frac{\eta(y)}{|y|^{d}} \ud y
+ \gamma^{0}_{\theta^{k}}-\gamma_{\theta^{k-1}}^{0}  \Big)\le \sum_{k=\bar k +1}^K
\Big( \int_{\hat A^k} \frac{\eta(y)}{|y|^{d}} \ud y
+ \gamma^{0}_{\theta^{k}}-\gamma_{\theta^{k-1}}^{0}  \Big)\\
&\le -(K - \bar k)\delta_{d,\lambda}.
\end{aligned}
\end{equation*}
Letting $K\to +\infty$, we deduce property (2).

\noindent
{\it Conclusion.} Up to a subsequence,  $\rho_n:=\chi_{E_n} \weakstar \rho$ for some $\rho$ in $L^1(\R^d;[0,1])$.
For every $\bar r\in (0,1)$, by Remark \ref{monosi} and Lemma \ref{rmonotonel1},  we have
$$
\begin{aligned}
\liminf_{n\to + \infty} \jj_{r_n}^{\s_n}(E_n) &  \ge 
\liminf_{n\to + \infty} \jj_{r_n}^{0}(E_n) 
\\
&  \ge  \liminf_{n\to + \infty}  \, \int_{U} \rho_n(x)j^{0}_{\bar r}(x,\rho_n)  \ud x
\\
& =    \, \int_{U} \rho(x) j^{0}_{\bar r}(x,\rho)\ud x,
\end{aligned}
$$
where in the last equality we have used that $\rho_n\weakstar \rho$ in $L^\infty(U)$ and that, by the dominated convergence Theorem, $j^{0}_{\bar r}(\cdot,\rho_n)\to j^{0}_{\bar r}(\cdot,\rho)$ in $L^1(U)$ as $n\to+\infty$\,.

Setting $\mathcal{N}:= \{x \in U : \, \rho(x) \in (0,1)\}$ and using the claims (1) and (2) we deduce that
$$
\begin{aligned}
C  \ge  \liminf_{n\to + \infty} \jj_{r_n}^{\s_n} (E_n) 
 \ge 
\lim_{\bar r \to 0} \int_{U} \rho(x)j^{0}_{\bar r}(x,\rho)  \ud x 
\ge  +\infty |\mathcal{N}| -\|\rho\|_{L^1}^2.
\end{aligned}
$$
As a consequence $\mathcal{N}$ is a negligible set, hence
$\rho$ is the characteristic function of some set $E\subset U$.  It  follows that $\chi_{E_n}\to \chi_E$ strongly in $L^1(\R^d)$.

The last part of the theorem is a trivial consequence of the monotonicity of $\jj^{\s}_r$ established in Lemma \ref{rmonotonel1}.
\end{proof}

\begin{remark}\label{infinitness}
\rm
By the proof of Theorem \ref{col}, and in particular by claims (1) and (2), it immediately follows that, for every $\s\in [0,1)$\,, 
$\jj^{\s}(\rho)=+\infty$ whenever $\rho$ is not the characteristic function of a set with finite measure. 
\end{remark}
We notice that the compactness property stated in Theorem \ref{col} is not satisfied by the functionals $\jj^{\s}$ for $\s\in (-d,0)$\,. In this case, indeed, $\jj^{\s}(\eta)$ is finite for every density $\eta\in L^1(\R^d;[0,1])$\,.
Nevertheless we have the following compactness result 
for the functionals $\jj^{\s}$ as $\s\to 0^{-}$\,.
\begin{theorem}\label{comp-}
Let $\{\s_n\}_{n\in\N}\subset (-d,0)$ and $\{r_n\}_{n\in\N}\subset \R^+$ be such that $\s_n\to 0^{-}$ and $r_n\to 0^+$\,.
Let $U\subset\R^d$ be an open bounded set and let  $\{E_n\}_{n\in\N}\subset \M$ be such that $E_n\subset U$ for all $n\in\N$\,.
Finally, let $C>0$\,.
 
If $\jj_{r_n}^{\s_n}(E_n)\le C$ for all $n\in\N$\,,
then, up to a subsequence, $\chi_{E_n} \to \chi_E$ in $L^1(\R^d)$ for some $E\in \M$.   

In particular, 
if  $\jj^{\s_n}(E_n)\le C$ for all $n\in\N$\,,
then, up to a subsequence, $\chi_{E_n} \to \chi_E$ in $L^1(\R^d)$ for some $E\in \M$.   
\end{theorem}
\begin{proof}
Up to a subsequence,  $\rho_n:=\chi_{E_n} \weakstar \rho$ for some $\rho$ in $L^1(\R^d;[0,1])$.
For every $\bar r\in (0,1)$, by Lemma \ref{rmonotonel1},  we have
$$
\begin{aligned}
\liminf_{n\to + \infty} \jj_{r_n}^{\s_n}(E_n) &  \ge  \liminf_{n\to + \infty}  \, \int_{U} \rho_n(x)j^{\s_n}_{\bar r}(x,\rho_n)  \ud x
\\
& =    \, \int_{U} \rho(x) j^{0}_{\bar r}(x,\rho)\ud x,
\end{aligned}
$$
where in the last equality we have used that $\rho_n\weakstar \rho$ in $L^\infty(U)$ and that, by the dominated convergence Theorem, $j^{\s_n}_{\bar r}(\cdot,\rho_n)\to j^{0}_{\bar r}(\cdot,\rho)$ in $L^1(U)$ as $n\to+\infty$\,.
By using claim (2) in the proof of Theorem \ref{col} and arguing as in the conclusion therein we get the statements.
\end{proof}

Finally, we prove the following compactness result also for the functionals  $\hh^{\s}_{R}$\,.
\begin{theorem}\label{compH}
Let $\{\s_n\}_{n\in\N}\subset (-d,1)$ and $\{R_n\}_{n\in\N}\subset (0,+\infty)$\,.
Let $U\subset\R^d$ be an open bounded set and let $\{E_n\}_{n\in\N}\subset \M$ be such that $E_n\subset U$ for all $n\in\N$\,. Finally, Let $C>0$\,.
If $\hh_{R_n}^{\s_n}(E_n)\le C$ for all $n\in\N$, we have:
 \begin{itemize}
\item[(a)] if $\{\s_n\}_{n\in\N}\subset [0,1)$, then, up to a subsequence, $\chi_{E_n} \to \chi_E$ in $L^1(\R^d)$ for some $E\in \M$,     
\item[(b)] if $\s_n\to 0$, then, up to a subsequence, $\chi_{E_n} \to \chi_E$ in $L^1(\R^d)$ for some $E\in \M$.    
\end{itemize}
\end{theorem}
\begin{proof}
By Proposition \ref{sconvdivR} we can assume without loss of generality that $R_n>\diam(U)$\,.
By Remark \ref{rmkbded}, we have that $\hh_{R_n}^{\s_n} (E_n)= \hh^{\s_n}(E_n)=\jj^{\s_n}(E_n)$\,. By Theorems \ref{col} and \ref{comp-} we deduce (a) and (b).
\end{proof}
\section{$\Gamma$-convergence}

This section is devoted to  the $\Gamma$-convergence analysis of the functionals $\jj^\s,\,\jj^\s_r,\,\hh^\s_R$ as $\s\to\bar\s$\,, $r\to \bar r$\,, $R\to\bar R$  for some $\bar\s\in (-d,1)$\,, $\bar r\in [0,+\infty)$\,, $\bar R\in (0,+\infty]$\,.

Next, we shall prove the $\Gamma$-convergence 
of the functionals $\hh_R^\s$ as $\s\to\bar\s$\,. 
Firstly, for smooth sets $E$\,, we show the pointwise convergence of $\hh_{R}^\s(E)$ to $\hh_{\bar R}^{\bar\s}(E)$ as $\s\to\bar\s$  and $R\to \bar R$ for some
 $\bar\s\in (-d,1)$ and $\bar R\in(0,+\infty]$\,.
From now on, it is convenient to adopt the notation $\hh^\s_\infty:= \hh^\s$.

\begin{proposition}\label{pointwiselimit}
Let $\bar\s\in (-d,1)$ and $\bar R\in (0,+\infty]$\,. Let moreover $\{\s_n\}_{n\in\N}\subset (-d,1)$ and $\{R_n\}_{n\in\N}\subset (0,+\infty]$ be such that $\s_n\to\bar \s$ and $R_n\to\bar R$ as $n\to+\infty$\,.
If $E\in\M$ is an open bounded set with boundary of class $C^2$\,,  then
\begin{equation}\label{perlimsup}
\lim_{n\to +\infty} \hh_{R_n}^{\s_n}(E)=\hh_{\bar R}^{\bar\s}(E)\,.
\end{equation}
\end{proposition}
\begin{proof}
We claim that
\begin{equation}\label{ptlimitR}
\lim_{\s\to \bar\s}\hh_R^{\s}(E)=\hh_R^{\bar\s}(E)\qquad\textrm{for every }R\in (0,+\infty]\,.
\end{equation}
Now we prove that \eqref{ptlimitR} implies \eqref{perlimsup}.
If $\bar R\in (0,+\infty)$\,, in view of \eqref{funraggi}, we have
\begin{equation}\label{uniforme1} 
\begin{aligned}
|\hh^{\s_n}_{R_n}(E)-\hh^{\s_n}_{\bar R}(E)|&\le |H^{\s_n}_{R_n}(E)-H^{\s_n}_{\bar R}(E)|+|\gamma^{\s_n}_{R_n}-\gamma^{\s_n}_{\bar R}||E|\\
&\le \Big |\int_{E}\int_{A_{R_n,\bar R}(x)}\frac{1}{|x-y|^{d+\s_n}}\ud y\ud x \Big| +|\gamma^{\s_n}_{R_n}-\gamma^{\s_n}_{\bar R}||E|\\
&= 2 |\gamma^{\s_n}_{R_n}-\gamma^{\s_n}_{\bar R}||E|\to 0\qquad\textrm{as } n\to +\infty \,,
\end{aligned}
\end{equation}
where $A_{R_n,\bar R}(x)$ denotes the annular ring centered at $x$ having as inner radius $\min\{R_n,\bar R\}$ and as outer radius $\max\{R_n,\bar R\}$\,.
Moreover, if $\bar R=+\infty$\,, then for $n$ large enough we have that  $R_n\ge 1$ and $\s_n\ge \hat\s$  for some $\hat\s\in (-d,1)$\,. Therefore, 
by Proposition \ref{-sconvdivR}, Proposition \ref{sconvdivR} and Proposition \ref{monotone}, for every $\e>0$ there exists $R_\ep>0$ such that, for every $\s\in (-d,1)$ and $R_n\ge R_\e$ we have
\begin{equation}\label{uniforme2}
|\hh^\s(E)-\hh^\s_{R_n}(E)|< \ep\,.
\end{equation}
By \eqref{uniforme1}, \eqref{uniforme2}, and \eqref{ptlimitR}, we get
$$
\lim_{n\to+\infty}\hh^{\s_n}_{R_n}(E)= \lim_{n\to +\infty}(\hh^{\s_n}_{R_n}(E)-\hh^{\s_n}_{\bar R}(E))+\lim_{n\to+\infty}\hh^{\s_n}_{\bar R}(E)=\hh^{\bar\s}_{\bar R}(E)\,,
$$
i.e., \eqref{perlimsup} holds.

Now we prove \eqref{ptlimitR} and we consider only in the case $\bar\s=0$, being the proof in the other cases fully analogous.
By \eqref{uniforme2} and   triangular inequality, for every $\e>0$ there exists $R_\ep>0$ such that, for every $\s\in (-d,1)$ and $R\ge R_\e$ we have
$$
|\hh^\s(E)-\hh^0(E)|<2\ep+|\hh^\s_{R}(E)-\hh_{R}^0(E)|\,.
$$
Therefore, in order to get \eqref{ptlimitR} for $R\in (0,+\infty]$, it is enough to prove it only for $R\in (0,+\infty)$\,.

{\it Claim: For every $\ep>0$ and for every $R>0$ there exists $\s_{\ep,R}$ with $|\s_{\ep,R}|>0$ such that 
\begin{equation*}
|\hh^\s_{R}(E)-\hh_{R}^0(E)|<\ep\qquad\textrm{ for all }\s\textrm{ with }|\s|<|\s_{\ep,R}|\,.
\end{equation*}}

In order to prove the claim, we preliminarily notice that
\begin{equation}\label{zero}
|\hh^\s_{R}(E)-\hh_{R}^0(E)|\le |\gamma_{R}^\s-\gamma^0_{R}| |E|+|H^\s_{R}(E)-H_{R}^0(E)| \, .
\end{equation}
As for the first addendum in \eqref{zero}, by the very definition of $\gamma_{R}^\s$ and $\gamma_{R}^0$ in \eqref{sgammaR} , we get
\begin{equation}\label{un}
|\gamma_{R}^\s-\gamma^0_{R}|=d\omega_d|\log R|\left|\frac{1-e^{-\s\log R}}{\s\log R}-1\right|\le |\s|d\omega_d \log^2 R\max \{R^\s,R^{-\s}\},
\end{equation}
where the inequality follows by applying 
$$
\left|\frac{1-e^{-t}}{t}-1\right|=\left|\frac{e^{-t}-(1-t)}{t}\right|\le \frac{t^2e^{|t|}}{|t|}=|t|e^{|t|}
$$
with $t=\s\log R$\,.
In order to estimate the second addendum in \eqref{zero}, we notice that
\begin{equation}\label{sum}
\begin{aligned}
|H^\s_{R}(E)-H_{R}^0(E)|&\le\int_{E}\int_{(B_{R}(x)\setminus E)\setminus B_1(x)}\left|\frac{1}{|x-y|^{d+\s}}-\frac{1}{|x-y|^{d}}\right|\ud y\ud x\\
&\phantom{\le}+\int_{E}\int_{(B_{R}(x)\setminus E)\cap B_1(x)}\left| \frac{1}{|x-y|^{d+\s}} - \frac{1}{|x-y|^d} \right|\ud y\ud x\, . 
\end{aligned}
\end{equation}
Notice also that, if $R\le 1$\,, the first integral in \eqref{sum} is equal to $0$\,, whereas, if $R>1$\,, in view of \eqref{un}, we have
\begin{equation}\label{sumhalf}
\begin{aligned}
&\int_{E}\int_{(B_{R}(x)\setminus E)\setminus B_1(x)}\left|\frac{1}{|x-y|^{d+\s}}-\frac{1}{|x-y|^{d}}\right|\ud y\ud x\\
\le &\int_{E}\int_{B_{R}(x)\setminus   B_1(x)}\left| \frac{1}{|x-y|^{d+\s}} - \frac{1}{|x-y|^{d}} \right|\ud y\ud x =& |E||\gamma_{R}^\s-\gamma_{R}^0|\\
\le& |\s|d\omega_d \log^2 R\max \{R^\s,R^{-\s}\} |E|\,,
\end{aligned}
\end{equation}
where the last equality follows by the fact that the integrand in the modulus has constant sign in the annulus $B_R(x)\setminus B_1(x)$\,.

In order to estimate the second integral in \eqref{sum}, we first consider the case $\s>0$\,.
Setting $\bar R:=\min\{1,R\}$ and $r_x:=\di(x,\partial E)$ for every $x\in E$\,,  we have 
\begin{equation}\label{deux}  
\begin{aligned}
&\int_{E}\int_{(B_{R}(x)\setminus E)\cap B_1(x)}\left| \frac{1}{|x-y|^{d+\s}} - \frac{1}{|x-y|^d} \right|\ud y\ud x\\
\le&\int_{E}\int_{B_{\bar R}(x)\setminus B_{r_x}(x)}\left(\frac{1}{|x-y|^{d+\s}}-\frac{1}{|x-y|^d}\right)\ud y\ud x
\\
\le &\s\int_{E}\int_{B_{\bar R}(x)\setminus B_{r_x}(x)}-\frac{\log|x-y|}{|x-y|^{d+\s}}\ud y\ud x\le \s d\omega_d\int_E\frac{1}{r_x^\s}\int_{r_x}^{\bar R}-\frac{\log\rho}{\rho}\ud\rho\ud x\\
\le&\s d\omega_d\frac 1 2\int_{E}\frac{1}{r_x^\s}\log^2r_x\ud x- \s d\omega_d\frac 1 2\frac{1}{\bar R^\s}\log^2\bar R|E|\,,
\end{aligned}
\end{equation}
where the first inequality follows by applying the bound (valid for $t\ge 1$)
$$
t^{d+\s}-t^d=t^{d+\s}(1-t^{-\s})\le t^{d+\s} \s\log t
$$
with $t=\frac{1}{|x-y|}$.

As for the case $\s<0$\,, by arguing as in \eqref{deux} one can easily show that
\begin{equation}\label{deux-}
\begin{aligned}
&\int_{E}\int_{(B_{R}(x)\setminus E)\cap B_1(x)}\left| \frac{1}{|x-y|^{d+\s}} - \frac{1}{|x-y|^d} \right|\ud y\ud x\le|\s| d\omega_d\int_{E}\log^2r_x\ud x-|\s|d\omega_d\log^2\bar R|E|\,.
\end{aligned}
\end{equation}

Therefore, in view of \eqref{zero}, \eqref{un}, \eqref{sum}, \eqref{sumhalf}, \eqref{deux}, \eqref{deux-},  the equality \eqref{ptlimitR} is proven once we show that there exists a constant $C(E)>0$ such that
\begin{equation}\label{finlog}
\int_{E}\frac{1}{r_x^{\s}}\log^2r_x\ud x\le C(E)\qquad\textrm{for every }\s\in [0,1)\,.
\end{equation}

Recalling that $E$ has boundary of class $C^2$, let $ T:= \frac 12 \|\Hsc\|_{L^\infty(\partial E)}^{-1}$, where $\Hsc$ is the second fundamental form.
Moreover, for all $t>0$ let 
$$
E_t:= \{x\in E : \, \di(x,\partial E) > t\}. 
$$
Then, for all $t\in(0,T)$ we have that $E_t$ has boundary of class $C^2$, and $\mathcal H^{d-1}(\partial E_t) \le C$ for some $C$ independent of $t$.
Then, by coarea formula we have 
$$
\begin{aligned}
\int_{E}\frac{1}{r_x^{\s}}\log^2r_x\ud x \le& 
\int_0^T \mathcal H^{d-1} (\partial E_t) \frac{1}{t^{\s}}\log^2 t \ud t
+
\Big( \frac{1}{T^{\s}} \max\{ \log^2 \diam(E), \log^2 T\} \Big)  |E_T|
\\
\le&
c_1 \int_0^T  \frac{1}{t^{\s}}\log^2 t \ud t + c_2 \le C(E),
\end{aligned}
$$
i.e., \eqref{finlog} holds.
\end{proof}
In the following proposition we show that a set $E\in\M$ with $\hh^\s(E)<+\infty$ can be approximated by a sequence of smooth sets. The same property related to the $\s$-fractional perimeters, with $\s>0$, has been proved in \cite{L}.
\begin{proposition}[Density of smooth sets]\label{density}
Let $\s\in (-d,1)$ and let $R\in (0,+\infty]$. Let $E\in \M$\,. If $\hh_R^\s(E) < +\infty$\,, then, there exists a sequence $\{E_n\}_{n\in\N}$ of bounded sets with smooth boundary such that $\chi_{E_n}\to \chi_E$ strongly in $L^1(\R^d)$ and $\hh_R^\s(E_n)\to \hh_R^\s(E)$ as $n\to +\infty$\,. 
\end{proposition}
\begin{proof}
First, we recall that for sets of finite perimeter, this result is classical, and its proof is based on the lower semicontinuity of the perimeter, on the convexity of the total variation functional, and on the coarea formula (see \cite[Theorem 3.42]{AFP}). Recalling that   $H^\s_R$ are lower semicontinuous and  that the functionals  $TV_{H^\s_R}$ introduced in \eqref{tv} are convex, the same proof shows that, for every $R\in (0,+\infty)$, there exists a sequence $\{E_{R,m}\}_{m\in\N}$ of bounded sets with smooth boundary such that $\chi_{E_{R,m}}\to \chi_E$ strongly in $L^1(\R^d)$ and $H^\s_R(E_{R,m})\to H_R^\s(E)$ as $m\to +\infty$. As a consequence, $\hh^\s_R(E_{R,m})\to \hh_R^\s(E)$ as $m\to +\infty$\,.
We now prove the statement for $R=+\infty$\,. 
By Propositions \ref{-sconvdivR} and \ref{sconvdivR},  the functionals $\hh^\s$ are lower semicontinuous; this fact, together with Propositions \ref{-sconvdivR},  \ref{sconvdivR}, and \ref{monotone},  implies
\begin{equation*}
 \hh^\s(E)\le 
    \liminf_{m\to +\infty} \hh^\s(E_{R,m})  
 \le   \liminf_{m\to +\infty} \hh_R^\s(E_{R,m}) = \hh^\s_R(E).
\end{equation*}
Since $ \hh^\s_R(E) \to  \hh^\s(E)$ as $R\to +\infty$, a standard diagonal argument provides a sequence $\{E_n\}_{n\in\N}$ with $E_n= E_{R_n,m_n}$ satisfying all the claimed properties.   
\end{proof}
We are now in a position to prove the $\Gamma$-convergence result for the functionals $\hh^\s_R$ as $\s\to \bar\s$ for some $\bar\s\in (-d,1)$, and $R\to \bar R$ for some $\bar R\in(0,+\infty]$\,.
\begin{theorem}\label{gammaconvHR}
Let $\bar R\in(0,+\infty]$ and  $\bar\s\in (-d,1)$. Let moreover $\{R_n\}_{n\in\N}\subset (0,+\infty]$ and $\{\s_n\}_{n\in\N}\subset (-d,1)$ be such that $R_n\to \bar R$ and $\s_n\to \bar\s$ as $n\to +\infty$\,. The following $\Gamma$-convergence result holds true.
\begin{itemize}
\item[(i)] ($\Gamma$-liminf inequality) For every $E\in\M$ and for every sequence $\{E_n\}_{n\in\N}$ with $\chi_{E_n}\to\chi_{E}$ strongly in $L^1(\R^d)$ it holds
\begin{equation*}
\hh^{\bar\s}_{\bar R}(E)\le\liminf_{n\to +\infty}\hh^{\s_n}_{R_n}(E_n)\,.
\end{equation*}
\item[(ii)] ($\Gamma$-limsup inequality) For every $E\in\M$\,, there exists a sequence $\{E_n\}_{n\in\N}$ such that $\chi_{E_n}\to\chi_{E}$ strongly in $L^1(\R^d)$ and
\begin{equation*}
\hh^{\bar\s}_{\bar R}(E)\ge \limsup_{n\to +\infty}\hh^{\s _n}_{R_n}(E_n)\,.
\end{equation*}
\end{itemize}
\end{theorem}
\begin{proof}
We first prove (i). We distinguish among two cases.

{\it Case 1: $\bar R\in (0,+\infty)$\,.}
Trivially, we have
\begin{equation}\label{triv}
\begin{aligned}
\liminf_{n\to +\infty}\hh_{R_n}^{\s_n}(E_n)-\hh^{\bar\s}_{\bar R}(E)\ge& \liminf_{n\to +\infty} H_{R_n}^{\s_n}(E_n)-H_{\bar R}^{\bar\s}(E)\\
&+\lim_{n\to +\infty}\gamma_{R_n}^{\s_n}|E_n|-\gamma_{\bar R}^{\bar\s}|E|\,.
\end{aligned}
\end{equation}
Moreover, by Fatou Lemma
\begin{equation*}\label{fatu}
 \liminf_{n\to +\infty} H_{R_n}^{\s_n}(E_n)\ge \int_{\R^d}\int_{\R^d}\liminf_{n\to +\infty}\chi_{B_{R_n}(x)}(y)\frac{\chi_{E_n}(x)(1-\chi_{E_n}(y))}{|x-y|^{d+\s_n}}\ud y\ud x=H_{\bar R}^{\bar\s}(E)\,,
\end{equation*}
which, together with \eqref{triv} and \eqref{un} implies (i).

{\it Case 2: $\bar R=+\infty$\,.} By  Theorem \ref{equfin}, Lemma \ref{rmonotone}, Lemma \ref{consn} below, and Propositions \ref{-sconvdivR}, \ref{sconvdivR}, \ref{monotone}, we have
$$
\begin{aligned}
\hh^{\bar\s}(E)&=\jj^{\bar\s}(E)=\lim_{r\to 0^+}\jj^{\bar\s}_{r}(E)=\lim_{r\to 0^+}\lim_{n\to +\infty}\jj^{\s_n}_{r}(E_n)\le \liminf_{n\to +\infty}\jj^{\s_n}(E_n)\\
&= \liminf_{n\to +\infty}\hh^{\s_n}(E_n)\le \liminf_{n\to +\infty}\hh_{R_n}^{\s_n}(E_n)\,,
\end{aligned}
$$
i.e., (i) holds.

Now we prove (ii). We can assume without loss of generality that $\hh^{\bar\s}_{\bar R}(E)<+\infty$\,.
If $E$ is smooth, in view of Proposition \ref{pointwiselimit}, in particular by \eqref{perlimsup}, the constant sequence $E_n\equiv E$ satisfies the $\Gamma$-limsup inequality.  The $\Gamma$-limsup inequality in the general case is an easy consequence  of Proposition \ref{density} and of  a standard diagunal argument, usually referred to as {\it density argument in $\Gamma$-convergence}. The details are left to the reader. 
\end{proof}
\begin{lemma}\label{consn}
Let $\bar\s\in (-d,1)$ and let $\bar r >0$. Let $\{\s_n\}_{n\in\N}\subset (-d,1)$ and $\{r_n\}_{n\in\N}\subset (0,+\infty)$ be such that $\s_n\to \bar\s$ and $r_n\to \bar r$  as $n\to +\infty$\,. Let moreover $E\in \M$ and $\{E_n\}_{n\in\N}\subset \M$ be such that $\chi_{E_n}\to \chi_{E}$ strongly in $L^1(\R^d)$ as $n\to +\infty$\,.
Then, 
\begin{eqnarray}\label{j}
&& j_{\bar r}^{\bar\s}(x,E)=\lim_{n\to +\infty}j_{r_n}^{\s_n}(x,E_n)\qquad\textrm{for every }x\in\R^d\\ \label{J}
&&\jj^{\bar\s}_{\bar r}(E)=\lim_{n\to +\infty}\jj^{\s_n}_{r_n}(E_n)\,.
\end{eqnarray}
\end{lemma}
\begin{proof}
We start by proving \eqref{j}.
Let $x\in\R^d$. It is easy to see that
\begin{equation}\label{prel}
\begin{aligned}
j^{\s_n}_{r_n}(x,E_n) = & \int_{\R^d\setminus B_{r_n}(x)}\frac{\chi_{E}(y)- \chi_{E_n}(y)}{|x-y|^{d+\s_n}}\ud y\\
&-\int_{\R^d\setminus B_{r_n}(x)} \frac{\chi_{E}(y)}{|x-y|^{d+\s_n}}\ud y-\gamma_{r_n}^{\s_n}\,.
\end{aligned}
\end{equation}
As for the first integral in \eqref{prel} we have
\begin{equation}\label{prel1}
\int_{\R^d\setminus B_{r_n}(x)}\frac{|\chi_{E_n}(y)-\chi_{E}(y)|}{|x-y|^{d+\s_n}}\ud y\le \frac{1}{{r_n}^{d+\s_n}}|E_n\Delta E|\to 0\,,
\end{equation}
while for the remaining terms, by the dominated convergence Theorem, we obtain
\begin{equation*}
-\int_{\R^d\setminus B_{r_n}(x)} \frac{\chi_{E}(y)}{|x-y|^{d+\s_n}}\ud y-\gamma_{r_n}^{\s_n}\to j^{\bar\s}_{\bar r}(x,E)\qquad\textrm{as }n\to +\infty\,,
\end{equation*}
which together with \eqref{prel}, and \eqref{prel1}, implies \eqref{j}.

Now we prove \eqref{J}.
In view of the strong $L^1$ convergence of the functions $\chi_{E_n}$ we have that there exists a constant $C>0$ such that
\begin{equation}\label{bd}
\sup_{n\in\N}|E_n|\le C\,.
\end{equation}
By \eqref{jjrexp}, we have
\begin{equation}\label{post}
\begin{aligned}
\jj^{\s_n}_{r_n}(E_n)&=\int_{\R^d}(\chi_{E_n}(x)-\chi_E(x))j_{r_n}^{\s_n}(x,E_n)\ud x+\int_{E}j_{r_n}^{\s_n}(x,E_n)\ud x\,.
\end{aligned}
\end{equation}
By \eqref{bd} we have
\begin{equation}\label{post1}
\begin{aligned}
\int_{\R^d}|\chi_{E_n}(x)-\chi_E(x)||j_{r_n}^{\s_n}(x,E_n)|\ud x\le \frac{C}{{r_n}^{d+\s_n}}|E_n\Delta E|\to 0\textrm{ as }n\to +\infty\,,
\end{aligned}
\end{equation}
whereas by \eqref{j} and by the dominated convergence Theorem we deduce
\begin{equation}\label{post2}
\begin{aligned}
\int_{E}j_{r_n}^{\s_n}(x,E_n)\ud x\to \jj_{\bar r}^{\bar\s}(E)\qquad\textrm{as }n\to +\infty\,.
\end{aligned}
\end{equation}
Therefore, \eqref{J} follows by \eqref{post}, \eqref{post1}, and \eqref{post2}.
\end{proof}

Finally we prove the $\Gamma$-convergence result for the functionals $\jj^\s_r$ as $\s\to \bar\s$ for some $\bar\s\in (-d,1)$ and $r\to \bar r$ for some $\bar r\in[0,+\infty)$\,. To this purpose, it is convenient to adopt the notation $\jj^{\s}_0:=\jj^\s$\,.
\begin{theorem}\label{gammaconvJsr}
Let $\bar\s\in (-d,1)$ and let $\bar r\in [0,+\infty)$\,. Let $\{\s_n\}_{n\in\N}\subset (-d,1)$ and $\{r_n\}_{n\in\N}\subset [0,+\infty)$ be such that $\s_n\to \bar\s$ and $r_n\to\bar r$ as $n\to +\infty$\,. The following $\Gamma$-convergence result holds true.
\begin{itemize}
\item[(i)] ($\Gamma$-liminf inequality) For every $E\in\M$ and for every sequence $\{E_n\}_{n\in\N}$ with $\chi_{E_n}\to\chi_{E}$ strongly in $L^1(\R^d)$ it holds
\begin{equation*}
\jj_{\bar r}^{\bar\s}(E)\le\liminf_{n\to +\infty}\jj_{\bar r_n}^{\s_n}(E)\,.
\end{equation*}
\item[(ii)] ($\Gamma$-limsup inequality)For every $E\in\M$\,, there exists a sequence $\{E_n\}_{n\in\N}$ such that $\chi_{E_n}\to\chi_{E}$ strongly in $L^1(\R^d)$ and
\begin{equation*}
\jj_{\bar r}^{\bar\s}(E)\ge \limsup_{n\to +\infty}\jj_{r_n}^{\s _n}(E)\,.
\end{equation*}
\end{itemize}
\end{theorem}
\begin{proof}
If $\bar r\in (0,+\infty)$\,, the statement follows immediately by \eqref{J}. We discuss the case $\bar r=0$\,. By Lemma \ref{rmonotone} and by \eqref{J} we have
$$
\jj^{\bar\s}(E)=\lim_{r\to 0^+}\jj^{\bar\s}_{r}(E)=\lim_{r\to 0^+}\lim_{n\to +\infty}\jj^{\s_n}_{r}(E_n)\le \liminf_{n\to +\infty}\jj_{r_n}^{\s_n}(E_n)\,,
$$
i.e., (i).
We prove (ii) for $\bar r=0$\,.
By  Theorem \ref{gammaconvHR}, Theorem \ref{equfin}, and Lemma \ref{rmonotone},  there exists a sequence $\{E_n\}_{n\in\N}\subset\M$ such that $\chi_{E_n}\to \chi_E$ as $n\to +\infty$ and
$$
\jj^{\bar\s}(E)=\lim_{n\to +\infty}\jj^{\s_n}(E_n)\ge \limsup_{n\to +\infty} \jj^{\s_n}_{r_n}(E_n)\,.
$$
\end{proof}

\section{The fractional isoperimetric inequality}
The isoperimetric inequality for the functionals $\jj^{\s}$ for $\s\in(-d,0)$ is nothing but the Riesz inequality (see \cite{Riesz} and Theorem \ref{riesz}). For $\s\in(0,1)$, one deals with fractional isoperimetric inequalities,  
that have  been proven in \cite{FS}, while their quantitative counterpart has been established in \cite{FMM} (see also \cite{FFMMM,DNRV}). 
Here we prove the (non quantitative) isoperimetric inequality and its stability also for the $0$-fractional perimeter. In fact, our short proof based on Riesz inequality yields the result for every exponent $\s\in [0,1)$.

Let $\s\in [0,1)$\,. For every $r>0$, we set
\begin{equation}\label{potri}
k^{\s}_r(t):=\frac{1}{\max\{t^{d+\s},r^{d+\s}\}} + (r-t)^+ \,,
\end{equation}
and we define the functionals $\Jsc_r^{\s}:\M\to(-\infty,0]$
as
\begin{equation*}
\Jsc_r^{\s}(E):=-\int_{E}\int_{E}k^{\s}_{r}(|x-y|)\ud y\ud x\qquad\textrm{ for all }E\in\M\,.
\end{equation*}
Notice that $k_r^{\s}$ is strictly decreasing with respect to $t$ and that, 
for every $E\in\M$,
\begin{equation}\label{Jhat}
\Jsc_r^{\s}(E)=J_r^{\s}(E) -  \int_{E} \int_{B_r(x)\cap E}   \frac{1}{r^{d+\s}} + (r-|x-y|)^+ \ud y     \ud x\,,
\end{equation}

We have the following result.

\begin{lemma}\label{leca}
Let $E, \, F \in \M$ with $|E|=|F|$, and, for $\s\neq0$, assume also that $E$ and $F$ have $C^2$ compact boundary. Then,  
\begin{equation}\label{bl}
\lim_{r\to 0} (\Jsc_r^{\s}(E) - \Jsc_r^{\s}(F)) =  
\lim_{r\to 0} (J_r^{\s}(E) - J_r^{\s}(F))=\jj^{\s}(E)-\jj^{\s}(F).
\end{equation}
\end{lemma}
\begin{proof}
We preliminarily notice that the second equality in \eqref{bl} is a trivial consequence of the very definition of $\jj^{\s}$ and of the fact that $|E|=|F|$\,.

Moreover, we notice that, for all $G\in\M$ we have
\begin{equation*}
\int_G \int_{B_r(x)\cap G} (r-|x-y|)^+ \ud y     \ud x=\delta_r\,,
\end{equation*}
with $\delta_r\to 0$ as $r\to 0^+$\,.
Therefore, by \eqref{Jhat} we have
\begin{equation}\label{diff}
\begin{aligned}
\Jsc_r^{\s}(E) -  \Jsc_r^{\s}(F)  & = J_r^{\s}(E)    -   J_r^{\s}(F) 
\\
& + \frac{1}{r^{d+\s}} \Big[
 \int_{F} |B_r(x)\cap F|  \ud x
-
 \int_{E} |B_r(x)\cap E|   \ud x\Big]+\delta_r\,,
\end{aligned}
\end{equation}
with $\delta_r\to 0$ as $r\to 0^+$\,.

We first consider the case $\s=0$\,. By the mean value and dominated convergence Theorems, for all $G\in \M$ we have 
\begin{equation}
\lim_{r\to 0^+}  \int_{G} \int_{B_r(x)\cap G}   \frac{1}{r^{d}} + (r-|x-y|)^+ \ud y \ud x = \omega_d |G|,
\end{equation}
whence, by \eqref{diff} we deduce \eqref{bl}.

Let now   $\s\in (0,1)$. Using the notation in \eqref{superl}, by coarea formula, for all $G \in \M$ with $C^2$ compact boundary  we have 
\begin{equation}\label{bl1}
\int_{{G\setminus G_r}} |B_r(x) \cap G| \ud x \le \omega_d r^d \int_0^r \mathcal H^{d-1} (\partial G_t)
\ud t \le C \omega_d r^{d+1}.
\end{equation}
Moreover, we have
\begin{equation}\label{bl2}
\begin{aligned}
&\Big|\int_{F_r} |B_r(x)\cap F_r|  \ud x
-
 \int_{E_r} |B_r(x)\cap E_r|   \ud x\Big|
=\omega_d r^d(|E_r|-|F_r|)\\
=&\omega_d r^d\Big||F\setminus F_r|-|E\setminus E_r|\Big|
\le \omega_d r^d\Big(|F\setminus F_r|+|E\setminus E_r|\Big)
\le C\omega_d r^{d+1}\,,
\end{aligned}
\end{equation}
where the last inequality easily follows by the coarea formula
and the regularity of $\partial E$, $\partial F$.

By \eqref{diff}, \eqref{bl1} and \eqref{bl2}, we deduce \eqref{bl} also for $\s\in (0,1)$\,.
\end{proof}

\begin{theorem}[Isoperimetric inequality]\label{isoteo}
For every $\s\in[0,1)$, the ball $B^m$ of measure equal to  $m>0$ is the unique, up to translations, minimizer of the $\s$-fractional perimeter $\jj^{\s}$ among all the measurable sets with measure equal to $m$. 

Moreover, if $\{E_n\}_{n\in\N}$ is a sequence of sets such that $|E_n|\equiv m$ and $\jj^{\s}(E_n) \to \jj^{\s}({B^m})$, then, there exists a sequence of translations $\{\tau_n\}_{n\in\N}$ such that $\chi_{E_n} + \tau_n \to \chi_{B^m}$ strongly in $L^1(\R^d)$. 
\end{theorem}

\begin{proof}
We set $\mathrm{Inf}:=\inf_{\newatop{E\in\M}{|E|=m}}\jj^{\s}(E)$\,; for all $\e>0$ let $E_\e$ (see Proposition \ref{density} and \cite{L} for $\s>0$) be a smooth set such that 
\begin{equation}\label{infimal}
\jj^{\s}(E_\e) - \mathrm{Inf} \le \e\,.
\end{equation}
Recalling the definition of $k_r^{\s}$ in \eqref{potri},  for every $0<r_1\le r _2\le 1$, we set 
$$
k_{r_1,r_2}(t):= k_{r_1}(t) - k_{r_2}(t) \qquad \text{ for all } t>0\,.
$$
Noticing that $k_{r_1,r_2}$ is monotonically non-increasing with respect to $t$, and using  Riesz inquality (Theorem \ref{riesz}) we have 
\begin{equation*}
\begin{aligned}
&   \Jsc^{\s}_{r_1}(E_\e) - \Jsc^{\s}_{r_1}(B^m) 
=
\Jsc^{\s}_{r_2}(E_\e) - \Jsc^{\s}_{r_1}(B^m) 
 -\int_{E_\ep}  \int_{E_\ep} k_{r_1,r_2}(|x-y|) \ud y \ud x
 \\
& \ge
\Jsc^{\s}_{r_2}(E_\e) - \Jsc^{\s}_{r_1}(B^m) 
 -\int_{B^m}  \int_{B^m} k_{r_1,r_2}(|x-y|) \ud y \ud x
\\
& =
\Jsc^{\s}_{r_2}(E_\e) - \Jsc^{\s}_{r_2}(B^m) =: c_{r_2}(\e)\ge 0\,, 
\end{aligned}
\end{equation*}
where the last inequality follows  again by Riesz inequality. 
Chosing $r_2=1$ and letting $r=r_1\to 0^+$, by Lemma \ref{leca} we deduce that
\begin{equation}
\begin{aligned}
 \jj^{\s}(E_\e) -  \jj^{\s}(B^m)   \ge c_1(\e),
\end{aligned}
\end{equation}
and by \eqref{infimal}  we conclude  that $c_1(\e)$ (and in fact $c_r(\e)$ for all positive $r$) vanishes
as $\e\to 0^+$. Therefore $\Jsc^{\s}_{1}(E_\e) \to \Jsc^{\s}_{1}(B^m)$ as $\ep\to 0^+$\,. Noticing that $k_1^\s$ is strictly decreasing, by Theorem \ref{stabiri} we deduce that, up to translations, $\chi_{E_\e} \to \chi_{B^m}$ strongly in $L^1(\R^d)$. The minimality of $B^m$ is then a consequence of the lower semicontinuity (together with the translational invariance) of $\jj^{\s}$.
\end{proof}
In view of \eqref{ovvieta} one may wonder whether both the functionals $H_1^\s$ and $J^\s_1$ are minimized, under volume constraints, by the ball.
We show that this is the case for $H^\s_1$ (see Proposition \ref{propsiH}) but, in general, not for $J^\s_1$ (see Remark \ref{remnoJ}).
\begin{remark}\label{remnoJ}
\rm
Let $\s\in (-d,1)$ and let $r>0$\,. We set $m_r:=\omega_d(\frac{r}{2})^d$\,. Clearly, for all $m\in (0,m_r)$ we have that $0=J_r^\s(B^{m})\ge J_r^{\s}(E)$ for all $E$ with $|E|=m$\,.
Moreover, taking $E:=B^{\frac m 2}\cup B^{\frac m 2}(\xi)$ with $|\xi|= 2r$\,, we have immediately that  $|E|=m$ and $J_r^{\s}(E)<0$\,. Therefore, for all $m\in (0,m_r)$ the ball  $B^m$ is a maximizer of $J^\s_r$\,; in particular, for general values of $m$ and $r$ the ball is not a solution of the isoperimetric inequality. 
\end{remark}
\begin{proposition}\label{propsiH}
Let $\s\in (-d,1)$ and let $R>0$\,. The ball $B^m$ of measure equal to  $m>0$ is the unique, up to translations, minimizer of  $H_R^{\s}$ among all the measurable sets with measure equal to $m$. 

Moreover, if $\{E_n\}_{n\in\N}$ is a sequence of sets such that $|E_n|\equiv m$ and $H_R^{\s}(E_n) \to H_R^{\s}({B^m})$as $n\to +\infty$\,, then, there exists a sequence of translations $\{\tau_n\}_{n\in\N}$ such that $\chi_{E_n} + \tau_n \to \chi_{B^m}$ strongly in $L^1(\R^d)$. 
\end{proposition}
\begin{proof} 
For every $0<r\le R$ we set $k^{\s}_{r,R}(t):=\chi_{[0,R]}(t)k^\s_{r}(t)$ where $k^\s_r$ is defined in \eqref{potri},
and we define $\K^\s_{r,R}:\M\to (-\infty,0)$ as
\begin{equation*}
\K^\s_{r,R}(E):=-\int_{\R^d}\int_{\R^d}\chi_E(x)\chi_E(y)k^{\s}_{r,R}(|x-y|)\ud y\ud x\,.
\end{equation*}
Let $E\in\M$ with $|E|=m$ and $H^\s_R(E)<+\infty$\,.
For every $0<r<\bar r<\min\{R,1\}$ we have
\begin{eqnarray}\nonumber
H_R^{\s}(E)-H_R^{\s}(B^m)&=&\K^\s_{\bar r,R}(E)-\K^\s_{\bar r,R}(B^m)\\ \label{decosob}
&&+(\K^\s_{r,R}(E)-\K^\s_{\bar r,R}(E))-(\K^\s_{r,R}(B^m)-\K^\s_{\bar r,R}(B^m))\\ \nonumber
&&+H_R^{\s}(E)-\K^\s_{r,R}(E)-H_R^{\s}(B^m)+\K^\s_{r,R}(B^m)\\ \label{decoso}
&\ge& \K^\s_{\bar r,R}(E)-\K^\s_{\bar r,R}(B^m)\\ \label{mandato}
&&+H_R^{\s}(E)-\K^\s_{r,R}(E)-H_R^{\s}(B^m)+\K^\s_{r,R}(B^m)\,,
\end{eqnarray}
where the non-negativity of the quantity in \eqref{decosob} follows by the monotonicity of $k_{r}^\s-k_{\bar r}^\s$ and by Riesz inequality. 

Now we show that the sum in \eqref{mandato} tends to $0$ as $r\to 0^+$\,.
Indeed, by the very definiton of $\K_{r,R}^\s$ and by \eqref{potri} we have
\begin{eqnarray}\nonumber
& H_R^{\s}(E)-\K^\s_{r,R}(E)-H_R^{\s}(B^m)+\K^\s_{r,R}(B^m)\\ \label{no1}
=&\displaystyle \int_{E}\int_{\R^d\setminus E}\frac{\chi_{B_R(x)}(y)}{|x-y|^{d+\s}}\ud y\ud x-\int_{E}\int_{\R^d\setminus E}{\chi_{B_R(x)}(y)}k_r^\s(|x-y|)\ud y\ud x\\ \label{no2}
&\displaystyle -\int_{B^m}\int_{\R^d\setminus B^m}\frac{\chi_{B_R(x)}(y)}{|x-y|^{d+\s}}\ud y\ud x+\int_{B^m}\int_{\R^d\setminus B^m}{\chi_{B_R(x)}(y)}k_r^\s(|x-y|)\ud y\ud x,
\end{eqnarray}
and $k_r^\s(t)$ monotonically incerases to $\frac{1}{t^{d+\s}}$ as $r\to 0^+$\,. By the monotone convergence Theorem, we have that the expressions in \eqref{no1} and \eqref{no2} tend to zero as $r\to 0^+$\,.
Therefore, by taking the limit as $r\to 0^+$ in \eqref{mandato} and by \eqref{decoso}, we get
\begin{equation}\label{ott}
H_R^{\s}(E)-H_R^{\s}(B^m)\ge  \K^\s_{\bar r,R}(E)-\K^\s_{\bar r,R}(B^m)\ge 0\,,
\end{equation}
where the last inequality follows by Theorem \ref{riesz}.
Noticing that $k^{\s}_{\bar r,R}$ is strictly decreasing in $(0,R)$ and using Proposition \ref{rieszprop}, we get that, up to translations, the ball $B^m$ is the unique minimizer of $\K^\s_{\bar r,R}$, and hence of $H_R^\s$\,.

Finally, if $\{E_n\}_{n\in\N}$ is a sequence of sets such that $|E_n|\equiv m$ and $H^{\s}_R(E_n) \to H^{\s}_R(\chi_{B^m})$, by \eqref{ott} we have that $\K^\s_{\bar r,R}(E_n)\to\K^\s_{\bar r,R}(B^m)$ as $n\to +\infty$\,. By Theorem \ref{stabiri}
we deduce that there exists a sequence of translations $\{\tau_n\}_{n\in\N}$ such that $\chi_{E_n} + \tau_n \to \chi_{B^m}$ strongly in $L^1(\R^d)$. 
\end{proof}
\appendix
\section{Rearrangement inequalities}
In this appendix we recall some results on rearrangement inequalities and we provide some cases of uniqueness and stability for the Riesz inequality, in the specific case of a set interacting with itself.

Let $K\in L^1_{\loc}(\R^d;[0,+\infty))$ be such that $K(z) = k(|z|)$ for some $k:[0,+\infty)\to [0,+\infty)$ monotonically non-increasing.
For every $\eta_1,\eta_2 \in L^1(\R^d;[0,+\infty))$ we set 
$$
I(\eta_1,\eta_2):= \int_{\R^d} \int_{\R^d} \eta_1(x) \eta_2(y) K(x-y) \ud y \ud x\,.
$$
First, we recall the classical Riesz inequality \cite{Riesz}. To this purpose, 
for every $m>0$ and $x_0\in\R^d$, we denote by $B^m(x_0)$  the ball centered in $x_0$ with $|B^m(x_0)|=m$ ($B^m$ if $x_0=0$). With a little abuse of notation, for any $x_0\in\R^d$ and  for any $\eta\in L^{1}(\R^d;[0,+\infty))$, we set $B^{\eta}(x_0):=B^{\|\eta\|_{L^1}}(x_0)$ ($B^{\eta}:=B^{\|\eta\|_{L^1}}$ if $x_0=0$).
Moreover, for every function $\eta\in L^1(\R^d;[0,+\infty))$  we denote by $\eta^*$ the spherical symmetric nonincreasing rearrangement of $\eta$, satisfying 
\begin{equation*}
\{\eta^{*} > t\} = B^{m_t} \text{ where } m_t:=|\{\eta>t\}| \qquad \text{ for all } t>0\,.
\end{equation*}
Now, we state the Riesz inequality, restricting our analysis to densities with values in $[0,1]$;
its proof is classical and we refer the reader to \cite{Bur}. 

\begin{theorem}[{Riesz inequality}]\label{riesz}
Let $\eta_1,\,\eta_2\in L^{1}(\R^d;[0,1])$ with $\|\eta_1\|_ {L^1},\|\eta_2\|_{L^1}>0$. Then, 
\begin{equation}\label{rieszineq}
I(\eta_1,\eta_2)\le I(\eta_1^*,\eta_2^*)
\le I(\chi_{B^{\eta_1}},\chi_{B^{\eta_2}})\,.
\end{equation}
Moreover, if $k$ is strictly decreasing, then the first inequality in \eqref{rieszineq} is an equality if and only if $\eta_i(\cdot)=\eta_i^*(\cdot-x_0)$ ($i=1,2$) for some $x_0\in\R^d$, whereas the second inequality in \eqref{rieszineq} holds with the equality if and only if $\eta_i^*=\chi_{B^{\eta_i}}$\,.
\end{theorem}

Equality cases have  been largely studied  in the literature (see \cite{L77, Bur2, CM, CN}); here we provide a case of equality specific for a characteristic function interacting with itself. 
 
For every $E\in\M$, we set $\K(E):= I(\chi_{E},\chi_E)$\,. 

\begin{proposition}[{An equality case}]\label{rieszprop}
Assume that  $k$ is strictly decreasing in a neighborhood of the origin.  If $E\in\M$ satisfies 
\begin{equation}\label{rieszeq}
\K(E) =  \K( B^{|E|})\,,
\end{equation}
Then $E= B^{|E|}(x_0)$ for some $x_0\in\R^d$.
\end{proposition}

\begin{proof}
 By the layer-cake principle,  we have
 $$
\K(E)= \int_0^{+\infty} \int_{\R^d} \int_{\R^d} \chi_E(x) \chi_{E}(y) \chi_{\{K>t\}} (x-y) \ud y \ud x \ud t\,.
 $$
By \eqref{rieszineq} and \eqref{rieszeq}  we have  that for a.e. $t>0$
$$
\int_{\R^d} \int_{\R^d} \chi_E(x) \chi_{E}(y) \chi_{\{K>t\}} (x-y) \ud y \ud x=\int_{\R^d} \int_{\R^d} \chi_{B^{|E|}}(x) \chi_{B^{|E|}}(y) \chi_{\{K>t\}} (x-y) \ud y \ud x\,.
$$

Set $\beta(t) := |{\{K>t\}}|$ for every $t$\,.  Since $K$ is radially symmetric and $k$ is monotonically  decreasing, we clearly have  that  ${\{K>t\}} = B^{\beta(t)}$
for all $t>0$. Moreover, since $k$ is strictly monotone in a neighborhood of the origin, we have that    for  all $\bar \beta >0$ the set $F_{\bar\beta}:=\{t>0\,:\, 0<\beta(t) <\bar \beta\}$ has positive measure. 
Furthermore, for a.e. $t\in F_{\bar\beta}$  we have
\begin{equation}\label{burcha}
\int_{\R^d} \int_{\R^d} \chi_{E}(x) \chi_{E}(y) \chi_{B^{\beta(t)}}(x-y) \ud x \ud y =  \int_{\R^d} \int_{\R^d} \chi_{B^{|E|}}(x)  \chi_{B^{|E|}}(y)  \chi_{B^{\beta(t)}}(x-y) \ud x \ud y\,.
\end{equation}  
Now fix $\bar \beta= 2 |E|$ and let  $t\in F_{\bar\beta}$ be such that \eqref{burcha} holds; by  \cite[Theorem 1]{Bur2} we conclude 
that,
up to a translation, $E= B^{|E|}$.
\end{proof}
We will also need the following result, whose proof is left to the reader.
\begin{proposition}\label{pull}
Let $\eta \in L^1(\R^d;[0,1])$ be such that $I(\eta,\eta) = I(\chi_{B^{\|\eta\|_{L^1}}},\chi_{B^{\|\eta\|_{L^1}}}) $\,. Then $\eta$ is a characteristic function. 
\end{proposition}
Now we provide a stability result for the Riesz inequality.  
\begin{theorem}[A stability case] \label{stabiri}
Assume  that $k$ is strictly decreasing in a neighborhood of the origin. Let $m>0$ and let $\{E_n\}_{n\in\N} \subset \M$ with $|E_n| \equiv m$ be such that $\K(E_n)\to \K(B^m)$. Then, there exists a sequence $\{\tau_n\}_{n\in\N}\subset \R^d$ such that $\chi_{E_n}(\cdot - \tau_n)\to \chi_{B^m}$ strongly in $L^1(\R^d)$. 
\end{theorem}
\begin{proof}
The proof is based on a concentration compactness argument {\it \`a la} Lions \cite{Lio}. We can assume without loss of generality that $m=1$\,.
Let $\{A^1_{n}\}_{n\in\N}, \,\{ A^2_{n}\}_{n\in\N}$ be two sequence of open sets and let $\lambda\in[\frac 12,1]$ be such that, up to a (not relabelled) subsequence   
\begin{equation*}
\begin{aligned}
& |E_n \cap A^1_{n}|  \to \lambda, \quad  |E_n \cap A^2_{n}|  \to 1- \lambda, \\
& \di(A^1_{n} , A^2_{n})\to +\infty \qquad \text{ as } n\to + \infty.
\end{aligned}
\end{equation*}
By Riesz inequality we have
\begin{equation*}
\K(B^1) = \limsup_{n\to +\infty}  \K(E_n) \le  \K({B^{\lambda}}) + \K({B^{1-\lambda}}), 
\end{equation*}
which clearly implies $\lambda=1$.
We have shown that there exists a subsequence for which it is impossible to split $E_n$ in two sets (with measure bounded away from zero) whose mutual distance diverges.  This, clearly implies the tight convergence of $\chi_{E_n}$ up to translations, 
 i.e., there exists a sequence of translations $\{\tau_n\}_{n\in\N}$ and a probability measure with density $\rho$ such that, up to a subsequence,
$\chi_{E_n}(\cdot-\tau_n)\weakstar \rho$ tightly. Since $\K$ is invariant by translations and continuous with respect to the tight convergence of characteristic functions, we deduce that 
$I(\rho,\rho)= \K({B^1})$, which together with 
Proposition \ref{pull}, yields $\rho=\chi_{E}$ for some $E\in\M$\,. By Proposition \ref{rieszprop} we get that $E$ is a ball, and hence the claim.
\end{proof}
We conclude with two lemmas that have been used in this paper. In these  results we replace the assumption $K\in L^1_{\loc}(\R^d;[0,+\infty))$ by the weaker assumption $K\in L^1_{\loc}(\R^d\setminus \{0\};[0,+\infty))$\,.
\begin{lemma}\label{rieszann0}
Let $R>0$ and  let $F\in\M$ with $F\subset B_R$.
Then
$$
\int_{B_R\setminus F}K(y)\ud y\ge \int_{B_R\setminus B^{|F|}} K(y)\ud y\,.
$$
\end{lemma}
\begin{lemma}\label{rieszannulus}
Let $s, \, m >0$. 
Then, for all $\rho\in L^1 (\R^d;[0,1])$ with $\|\rho\|_{L^1} \le m$ and with supp$(\rho) \subseteq \R^d\setminus B_s$, we have
$$
\int_{\R^d} \rho(y) K(y) \ud y \le \int_{A_{s,R(m,
s)}}  K(y)  \ud y,
$$  
where $A_{s_1,s_2}$ denotes the annulus $B_{s_2}\setminus B_{s_1}$  for all $0<s_1<s_2$, and $R(m,s)= (\frac {m}{\omega_d} + s^d)^{\frac 1d}$ (so that  $|A_{s, R(m,s)}| = m$)\,.
\end{lemma}
The proofs of Lemmas \ref{rieszann0} and \ref{rieszannulus} are easy consequences of standard rearrangement techniques and are left to the reader.



\begin{thebibliography}{99}

\bibitem{ADPM}
L. Ambrosio, G. De Philippis, L. Martinazzi: $\Gamma$-convergence of nonlocal perimeter functionals, {\it Manuscripta Math.} {\bf 134} (2011), no. 3, 377--403.

\bibitem{AFP}
L. Ambrosio, N. Fusco, D. Pallara: {\it Functions of Bounded Variation and Free Discontinuity Problems}. Clarendon Press, Oxford (2000).

\bibitem{BrTr}
A. Braides, L. Truskinovsky: Asymptotic expansions by $\Gamma$-convergence,
{\it Contin. Mech. Thermodyn.}  {\bf 20} (2008), no. 1, 21--62.

\bibitem{Bur2}
A. Burchard: Cases of equality in the Riesz rearrangement inequality, {\it Ann. of Math. (2)} {\bf 143} (1996), no. 3, 499--527.

\bibitem{Bur}
A. Burchard: {\it A short course on rearrangement inequalities}, lecture notes (2009). 

\bibitem{CRS}
L. Caffarelli, J.M. Roquejoffre, O. Savin: Non-local minimal surfaces, {\it Comm. Pure Appl. Math.} {\bf 63} (2010), 1111--1144.

\bibitem{CM}
E. Carlen, F. Maggi: Stability for the Brunn-Minkowski and Riesz rearrangement inequalities, with applications to Gaussian concentration and finite range non-local isoperimetry, {\it Canad. J. Math.} {\bf 69} (2017), no. 5, 1036Ð1063. 

\bibitem{CN}
A. Cesaroni, M. Novaga: The isoperimetric problem for nonlocal perimeters,
{\it Discrete Contin. Dyn. Syst. Ser. S} {\bf 11} (2018), no. 3, 425--440.

\bibitem{CGL}
A. Chambolle, A. Giacomini, L. Lussardi: Continuous limits of discrete perimeters, {\it M2AN Math. Model. Numer. Anal.} {\bf 44} (2010), no. 2, 207--230.

\bibitem{CMP}
{A. Chambolle, M. Morini, M. Ponsiglione:} Nonlocal curvature flows, {\it Arch. Ration. Mech. Anal.} {\bf 218} (2015), no. 3, 1263--1329.

\bibitem{CW}
H. Chen, T. Weth: The Dirichlet Problem for the Logarithmic Laplacian, to appear in {\it Comm. PDE}.

\bibitem{AdP}
E. Correa, A. de Pablo: Remarks on a nonlinear nonlocal operator in Orlicz spaces,
{\it Adv. Nonlinear Anal.} {\bf 9} (2020), 305--326.

\bibitem{DalM}
G. Dal Maso: {\it An introduction to $\Gamma$-convergence},  Progress in Nonlinear Differential Equations and their Applications, {\bf 8}, Birkh\"auser Boston, Inc., Boston, MA (1993).

\bibitem{DNRV}
A. Di Castro, M. Novaga, B. Ruffini, E. Valdinoci: Nonlocal quantitative isoperimetric inequalities,
{\it Calc. Var. Partial Differential Equations} {\bf 54} (2015), no. 3, 2421--2464.

\bibitem{DFPV}
S. Dipierro, A. Figalli, G. Palatucci, E. Valdinoci: Asymptotics of the $s$-perimeter as $s\searrow 0$, 
{\it Discrete Cont. Dyn. Syst.} {\bf 33} (2013), no. 7, 2777--2790.

\bibitem{FFMMM}
A. Figalli, N. Fusco, F. Maggi, V. Millot, M. Morini: Isoperimetry and stability properties of balls with respect to nonlocal energies, {\it Comm. Math. Phys.} {\bf 336} (2015), no. 1, 441--507. 

\bibitem{FS}
R.L. Frank, R. Seiringer: Non-linear ground state representations and sharp Hardy inequalities, {\it J. Funct. Anal.} {\bf 255} (2008), 3407--3430.

\bibitem{FMM}
N. Fusco, V. Millot, M. Morini:A quantitative isoperimetric inequality for fractional perimeters, {\it J. Funct. Anal.} {\bf 261} (2011), no. 3, 697--715.

\bibitem{I} C. Imbert. Level set approach for fractional mean curvature flows. {\it Interfaces Free Bound.} 
{\bf 11} no. 1  (2009), 153--176.

\bibitem{JW}  
S. Jarohs, T. Weth: Local compactness and nonvanishing for weakly singular nonlocal quadratic forms, to appear in {\it Nonlinear Anal.}.

\bibitem{L77}
E. H. Lieb: Existence and uniqueness of the minimizing solution of Choquard's nonlinear equation, {\it Studies in Appl. Math.} {\bf 57} (1976/77), no. 2, 93--105.

\bibitem{Lio}
P.-L. Lions: The concentration-compactness principle in the calculus of variations. The locally compact case. I., {\it Ann. Inst. H. Poincar\'e Anal. Non Lin\'eaire} {\bf 1} (1984), 109--145.

\bibitem{L} 
L. Lombardini: Approximation of sets of finite fractional perimeter by smooth sets and comparison of local and global $s$-minimal surfaces.  
{\it Interfaces Free Bound.} {\bf 20} (2018), no. 2, 261--296. 

\bibitem{MS}
V. Maz'ya, T. Shaposhnikova : On the Bourgain, Brezis, and Mironescu theorem concerning limiting embeddings of fractional Sobolev spaces, {\it J. Funct. Anal.} {\bf 195} (2002), 230--
238; Erratum, {\it J. Funct. Anal.} {\bf 201} (2003), 298--300.

\bibitem{Riesz}
F. Riesz: Sur une in\'egalit\'e int\'egrale. 
{\it Journ. London Math. Soc.} {\bf 5}, 162--168, 1930.


\bibitem{Vis1}
A. Visintin: Nonconvex functionals related to multiphase systems, {\it SIAM J. Math. Anal.} {\bf 21} (1990), no. 5, 1281--1304.

\bibitem{Vis2}
A. Visintin: Generalized coarea formula and fractal sets, {\it Japan J. Indust. Appl. Math.} {\bf 8} (1991), no. 2, 175--201.


\end{thebibliography}
\end{document}